\documentclass{amsart}
\usepackage{amssymb}
\usepackage{amscd}
\usepackage{amsmath}
\usepackage{enumerate}
\usepackage{float}
\usepackage[T1]{fontenc} 
\usepackage[english]{babel} 
\usepackage{amsmath,amsfonts,amsthm} 
\usepackage{fancyhdr} 
\usepackage{verbatim} 
\usepackage[hyperindex,breaklinks]{hyperref} 
\hypersetup{          
    colorlinks=true,
    linkcolor=blue,
    filecolor=magenta,      
    urlcolor=blue,
    citecolor=magenta,
    linktoc=all,
}
\usepackage[nameinlink]{cleveref}
\usepackage[normalem]{ulem}  
\usepackage{mathtools}
\PassOptionsToPackage{linktocpage}{hyperref}

\usepackage{graphicx} 
\usepackage[font=footnotesize,labelfont=bf]{caption} 
\usepackage{tikz-cd} 
\usepackage{setspace}

\usepackage{xcolor}

\numberwithin{equation}{section} 
\numberwithin{figure}{section} 
\numberwithin{table}{section} 

\theoremstyle{definition}
\newtheorem{remark}[equation]{Remark}
\newtheorem{notation}[equation]{Notation}
\crefname{notation}{notation}{notation}
\newtheorem{example}[equation]{Example}
\newtheorem{definition}[equation]{Definition}

\theoremstyle{plain}
\newtheorem{theorem}[equation]{Theorem}
\newtheorem{lemma}[equation]{Lemma}
\newtheorem{corollary}[equation]{Corollary}
\newtheorem{proposition}[equation]{Proposition}

\makeatletter
\let\c@equation\c@figure
\makeatother

\newtheoremstyle{named}{}{}{\itshape}{}{\bfseries}{.}{.5em}{#1 \thmnote{#3}}
\theoremstyle{named}
\newtheorem*{theorem*}{Theorem}

\theoremstyle{named}
\newtheorem*{corollary*}{Corollary}

\DeclareMathOperator{\lcm}{lcm} 

\newcommand{\R}{\mathbb{R}}
\newcommand{\Q}{\mathbb{Q}}
\newcommand{\C}{\mathbb{C}}
\newcommand{\N}{\mathbb{N}}

\newcommand{\Z}{\mathbb{Z}}

\newcommand{\Si}{\Sigma}

\newcommand{\G}{\Gamma}
\newcommand{\tat}{t\^ete-\`a-t\^ete }
\newcommand{\Tat}{T\^ete-\`a-t\^ete }

\newcommand{\MS}{\mathbb{S}}
\newcommand{\TG}{\widetilde{\G}}

\newcommand{\cont}{\subseteq}
\newcommand{\comp}{\circ}
\newcommand{\calC}{\mathcal{C}}
\newcommand{\calA}{\mathcal{A}}
\newcommand{\calB}{\mathcal{B}}

\newcommand{\calD}{\mathcal{D}}
\newcommand{\calU}{\mathcal{U}}

\usepackage{color}

\begin{document}
\title[T\^ete-\`a-t\^ete twists]{{T\^ete-\`a-t\^ete twists, monodromies and representation of elements of Mapping Class Group}}
\author{N. A'Campo}
\address{University of Basel. Spiegelgasse 1, CH-4051 Basel, Switzerland}
\email{Norbert.ACampo@unibas.ch}
\author{J. Fern\'andez de Bobadilla}
\address{ IKERBASQUE, Basque Foundation for Science, Maria Diaz de Haro 3, 
48013, Bilbao, Basque Country, Spain
(2) BCAM - Basque Center for Applied Mathematics, Mazarredo 14, 48009 Bilbao, Basque Country,  Spain}
\email{jbobadilla@bcamath.org} 
\author{M. Pe Pereira}
\address{Facultad de Ciencias Matematicas- UCM, Plaza de Ciencias, 3,  Ciudad 
Universitaria, 28040 Madrid, Spain} 
\email{maria.pe@mat.ucm.es}
\author{P. Portilla Cuadrado}
\address{(1)ICMAT,  Campus Cantoblanco UAM, C/ Nicol\'as Cabrera, 13-15, 28049 Madrid, Spain
(2) BCAM,  Basque Center for Applied Mathematics, Mazarredo 14, 48009 Bilbao, Basque Country,  Spain}
\email{p.portilla89@gmail.com}
\thanks{{First author thanks Professors Makoto Sakuma and Ichiro Shimada for their invitation to expose about monodromy and t\^ete-\`a-t\^ete twists at the conference "Branched Coverings, Degenerations and Related Topics, 1916," at the Graduate School of Sciences, Hiroshima University. Together with second author we thank Professors Yakov Eliashberg, David Nadler and Laura Paul Starkston for their invitation to present in the conference "Arborealization of singularities of Lagrangian skeleta, 1918," at the American Institut for Mathematics, San Jose, higher dimensional symplectic extensions. } Second author supported by ERCEA 615655 NMST Consolidator Grant, MINECO by the project reference MTM2013-45710-C2-2-P, by the Basque Government through the 
BERC 2014-2017 program, by Spanish Ministry of Economy and Competitiveness MINECO: BCAM Severo Ochoa excellence accreditation SEV-2013-0323 and by Bolsa Pesquisador 
Visitante Especial (PVE) - Ci\^encia sem Fronteiras/CNPq Project number: 401947/2013-0. Part of this research was developed while being a member at IAS; he thanks IAS for 
excellent working conditions and AMIAS for support.}
\thanks{Third author is supported by MINECO by the project with reference MTM2017-89420-P and  by ERCEA 615655 NMST Consolidator Grant  and by Bolsa Pesquisador 
Visitante Especial (PVE) - Ci\^encias sem Fronteiras/CNPq Project number: 401947/2013-0.  Part of this research was developed while being supported by a Juan de la Cierva Incorporaci\'on Grant from MINECO at ICMAT,  by ICMAT Severo Ochoa excellence accreditation SEV-2015-0554 and by MINECO project MTM2013-45710-C2-2-P.  }
\thanks{Fourth author supported by SVP-2013-067644 Severo Ochoa FPI grant and by project by MTM2013-45710-C2-2-P, the two of them by MINECO; also supported by the project ERCEA 615655 NMST Consolidator Grant}

\subjclass[2010]{32S40, 57M50 (primary), and
32B10, 32S05 (secondary)}



\begin{abstract}
We study monodromies of plane curve singularities and pseudo-periodic homeomorphisms of oriented
	surfaces with boundary, following an original idea of the first author: \tat graphs and
	twists. We completely characterize mapping classes that can be represented
	by \tat twists, and generalize the notion
	to be able to represent any class of the mapping class group relative to the
	boundary which is boundary-free periodic. This improves previous work on the subject by C. Graf. Furthermore, we introduce the class
	of mixed \tat graphs and twists, and prove that mixed \tat twists
	contain monodromies of 
	irreducible plane curve singularities. In a sequel paper, the fourth author and B. Sigurdsson have extended this to the reducible case.
\end{abstract}

\maketitle

\addtocontents{toc}{\protect\sloppy}
\tableofcontents

\section{Introduction}

Max Dehn has introduced the so called Dehn twist. More precisely, given an embedded copy $\alpha$ of the circle in the interior an oriented surface $S$, he has defined a mapping class $D_\alpha$ of the surface $S$.
The mapping class $D_\alpha$ has for every open subset $U$ that contains $\alpha$ a representative $\delta_U$ with support in $U$, that leaves $\alpha$ invariant and whose restriction to $\alpha$ has order $2$.

The first author has observed that the geometric monodromy of a Pham singularity $x^a+y^b$ in two complex variables is a "twist" attached in  rather similar manner to the Pham graph $Ph_{a,b}$ in the Milnor fiber $F_{a,b}$. More precisely, the graph $\Gamma=Ph_{a,b}$ is the complete bi-coloured graph with $a$ vertices of one colour and $b$ vertices of the other colour and is embedded
into the oriented surface $F_{a,b}$ such that the
t\^ete-\`a-t\^ete property holds with respect to a metric. This property (see \Cref{def:tat}), allows to define a mapping class $D_\Gamma$ of $F_{a,b}$ with a representative that is the identity away from $\Gamma$ and whose restriction to $\Gamma$ has finite order (\Cref{def:phi_Gtau}). This class is called a \tat twist and is a generalization of Dehn twists.

A first optimism suggests that all geometric monodromies of plane curve singularities are t\^ete-\`a-t\^ete twists attached to a graph in the Milnor fiber. This is particularly nice since the \tat graph is a vanishing spine of the Milnor fibre: is the part of the Milnor fibre that collapses approaching the singular fibre. Then the vanishing spine, together with an additional metric structure codifies the geometric monodromy. However, this is only true for finite order monodromies. 

In fact it is natural to explore which mapping classes are \tat twists. This 
task is completed in the present paper, where the following result is proved 
(\Cref{cor:original_def}): given any oriented surface $\Sigma$ with boundary, 
any mapping class $\phi\in MCG(\Sigma,\partial\Sigma)$ is a \tat twist if and 
only if $\phi$ is boundary-free isotopic to a periodic homeomorphism, and the 
fractional Dehn twist coefficients (see \Cref{def:gen_rot}) at each of the 
components of the boundary are positive. In fact a more general result is 
proven (see \Cref{thm:000} and \Cref{thm:rel_signed}): by extending the 
definitions to that of \emph{signed relative \tat graphs and twits}, if 
$\partial^1\Sigma$ is a non-empty union of boundary components, any mapping 
class $\phi\in MCG(\Sigma,\partial^1\Sigma)$ is a signed relative \tat twist if 
and only if $\phi$ is boundary-free isotopic to a periodic homeomorphism. This 
results improve results of Graf~\cite{Graf1,Graf}, who proved the analogous 
realization, but using a class of graphs and twists that strictly contains 
ours. We also care of representing all periodic boundary-free mapping classes 
by \tat twists (\Cref{thm:rel}).  

The above optimism on geometric monodromies becomes a Theorem, if one works 
with graphs that satisfy the mixed t\^ete-\`a-t\^ete property 
(\Cref{def:relative_mix_sim}), and with the mapping classes defined by them, 
which are called mixed \tat twists and are also introduced in this paper. { The 
proof of this Theorem is the conjunction of the results of the second part of 
this paper and a subsequent paper by the fourth author and B. 
Sigurdsson~\cite{Bal}}. Here we prove that the monodromy of any irreducible 
plane branch is a mixed \tat twist (this is a consequence of the more general 
\Cref{thm:mixed_model}): 

{ we introduce a filtered vanishing spine on a surface, together with a metric 
structure (which is called a mixed t\^ete-\`a-t\^ete graph), associate to it a 
mixed t\^ete-\`a-t\^ete twist. In the case of irreducible plane curve 
singularities we find a mixed t\^ete-\`a-t\^ete graph on the Milnor fibre, 
whose associated mixed t\^ete-\`a-t\^ete twist is the geometric monodromy. 
In~\cite{Bal}, mixed \tat twists has been characterized as the geometric 
monodromy for an isolated curve singularity $f:X\to \mathbb{C}$ on a normal 
surface germ, covering in particular the reducible plane curve case.} 

{The plane curve case suggested the following higher dimensional generalization to the first author}. The Pham spine $Ph_{a_0, \cdots ,a_n},\, n>1,$ in the Milnor fiber of $z_0^{a_0}+ \cdots + z_n^{a_n}$ has also a t\^ete-\`a-t\^ete property if one realizes the simplices
of the spine as orthogonal spherical simplices. {
$\pi$-Walks are broken spherical geodesics }. Again the monodromy is the corresponding twist. Taking into account the natural symplectic structure on the Milnor fibre one finds that the Pham spine is a Lagrangian vanishing spine, in the sense explained above. Such a  symplectic interpretation allows an optimism about the geometric monodromy of complex hypersurface singularities. This topic needs however further investigation in future.

The gist for the ideas in this paper is contained in the first author preprint~\cite{Camp1}. The present paper fully develop the ideas contained there for the surface case and systematically studies mapping classes in terms of \tat graphs. 

We have tried to make the paper as self-contained as possible. The structure is as follows. In \Cref{sec:thick} we introduce the basic definitions on ribbon graphs and thickening surfaces, and set up the notation and terminology for the rest. In \Cref{00puretat} we define \tat graphs and prove their basic properties. In \Cref{sec:iso} we recall basic facts on periodic mapping classes, and introduce the fractional Dehn twist coefficients. In \Cref{sec:tatatat} \tat twists are introduced, and the characterization of \tat twists explained above is proved. In \Cref{sec:pseudo} we recall the facts we need on pseudo-periodic homeomorphisms and set up the notation for the rest of the paper. In \Cref{sec:mixed} mixed \tat graphs and twists are introduced, and the realization theorem explained above, which implies that monodromies of 
irreducible branches are \tat twists, is  proved. In this last section, for pedagogical reasons, we start by analyzing a class of pseudo-periodic homeomorphisms, and constructing a spine on the underlying surface with special properties. This motivates the definition of mixed \tat graphs and twists. After this previous analysis the proof of main \Cref{thm:mixed_model} follows.

\section{Graphs, spines and regular thickenings}\label{sec:thick}

A \emph{graph} $\G$ is a $1$-dimensional finite CW-complex. We denote by $v(\G)$ the \emph{set of vertices} and by 
$e(\G)$  the \emph{set of edges}. We allow  \emph{loops} (edges starting and ending in the same vertex) 
and also several edges connecting two vertices. 
For a vertex $v$ we denote by $e(v)$ the set of edges adjacent to $v$, where an edge $e$ appears twice in case it is a 
loop joining $v$ with $v$. The \emph{valency} of a vertex is the cardinality of $e(v)$. Unless we state the contrary we assume that there are no vertices of valency $1$. 

A {\em ribbon graph} is a graph $\G$ such that for every vertex we fix a cyclic order in the set of edges $e(v)$. 

A \emph{regular thickening} of a ribbon graph $\G$ is a piecewise smooth embedding $\G\hookrightarrow \Si$  of the 
graph, as a deformation retract of an
oriented surface with boundary, such that the cyclic ordering of the incoming edges at each vertex is induced 
counterclockwise by the orientation of the surface. The thickening surface $(\Si, \G)$ is unique only up to orientation 
preserving 
homeomorphism of the surface.

Reciprocally, every oriented surface with finite topology and non-empty boundary has a \emph{spine} $\G$
(i.e. an embedded graph in $\Si\setminus\partial \Si$ that is
a \emph{regular retract} of $\Si$) 
with a ribbon graph structure whose regular thickening is $(\Sigma,\G)$. 

\begin{example}\label{ex:Kpq}
Let $K_{p,q}$ be the bipartite graph $(p,q)$. The set of vertices is the union of two sets $A$ and $B$ of $p$ and $q$ vertices respectively. The edges are exactly all the 
possible non-ordered pairs of points one in $A$ and one in $B$.  

Now we fix cyclic orderings in $A$ and $B$. These give cyclic orderings in the sets of edges adjacent to vertices in $B$ and $A$ respectively. 

One can check that the thickening surface has as many  boundary components as $\gcd(p,q)$ and genus equal to $\frac{1}{2}[(p-1)(q-1)- \gcd(p,q)+1]$. 

In \Cref{fig:bipartite} we have the example of $K_{2,3}$. 

\begin{figure}[!ht]
\centering
\includegraphics[scale=0.5]{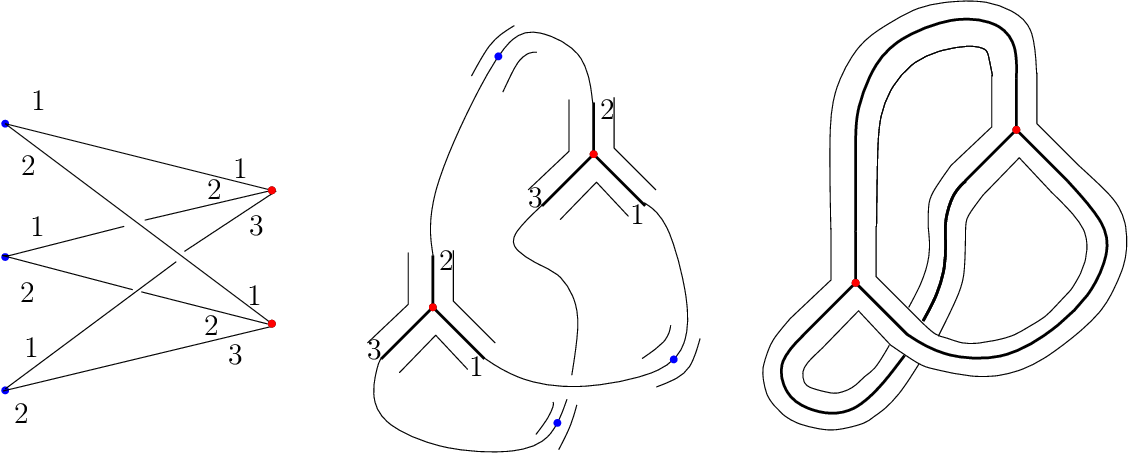}
\caption{Thickening of the graph $K_{2,3}$ in three steps. First we have a planar projection of $K_{2,3}$ where the two 
subsets of 2 and 3 vertices are vertically ordered in 
different parallel lines. Then we thicken a neighbourhood of every vertex and finally we glue the pieces. The resulting surface is homeomorphic to the once-punctured torus.} 
\label{fig:bipartite}
\end{figure}
\end{example}

We introduce a generalization of the notion of spine of a surface with boundary $\Si$, 
which treats in a special way a certain union of boundary components. Let us start by the corresponding graph theoretic notion.

\begin{definition}
\label{def:rel_ribbon} 
Let $(\G,A)$ be a pair formed by a graph $\G$ and an oriented subgraph $A$ such that each of its connected components $A_i$ is homeomorphic to the oriented circle
$\mathbb{S}^1$. The pair $(\G,A)$ is a \emph{relative ribbon graph} if for any vertex $v \in A$ the set of incident edges
$e(v)$ is endowed with a cyclic ordering $e(v)=\{e_1,...,e_k\}$ \textit{compatible} with the orientation of $A$, which means that 
\begin{itemize}
 \item $e_i<e_{i+1}$ and $e_k<e_1$,
 \item the edges belonging to $A$ (that necessarily belong to a single component $A_i$) are $e_1$ and $e_k$,
 \item if we consider a small interval in $A_i$ around $v$ and we parametrize it in the direction induced by the orientation of $A_i$, then  we pass first by $e_1$ and after by 
 $e_k$.
 \end{itemize}
 \end{definition}


A \emph{regular thickening} of a relative ribbon graph $(\Gamma,A)$
 is a piecewise smooth embedding of pairs
$(\Gamma,A)\hookrightarrow (\Sigma,\partial\Sigma)$ into a surface $\Sigma$ with boundary
such that each $A_i$ is sent homeomorphically (orientation reversing) to a component of $\partial\Sigma$, the subgraph $\Gamma\setminus A$ is 
sent to the interior of $\Sigma$, and the image of $\Gamma$ is a deformation retract of $\Sigma$. 
The cyclic ordering of the incoming edges at each vertex is induced counterclockwise by the orientation of the surface. 
The thickening surface $(\Sigma,\G,A)$ is unique only up to orientation preserving homeomorphism of the surface.

%
%
%

\begin{figure}[!ht]
\centering
\includegraphics[scale=0.5]{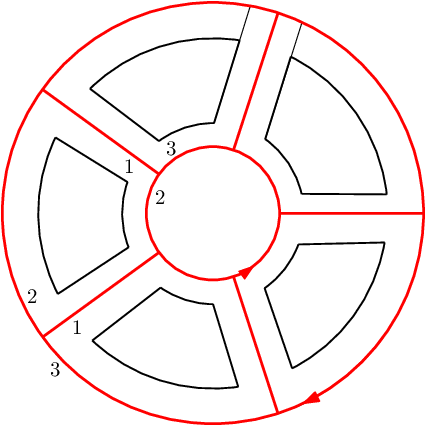}
\caption{Thickening of a relative ribbon graph with two relative components. The arrows in the picture indicate the direction that safe walk takes in the relative components. The numbers indicate the cyclic ordering of the edges adjacent to two different vertices. The surface is a sphere with $7$ holes .}
\label{fig:rel}
\end{figure}

Reciprocally: for any pair $(\Si,A)$, given by an oriented surface $\Sigma$ and a union of some boundary 
components $A$, if there is a graph $\G$ embedded in $\Si$ with $A\subset \G$ 
 such that $\G$ is a regular retract of $\Si$, then $(\Si,\G, A)$ is a thickening of $(\G,A)$.

\begin{notation}\label{not:cut}
From now on the letter $I$ denotes an interval, unless  otherwise specified, it denotes the unit interval.
Let $(\G,A)$ be a relative ribbon graph and $(\Si,\G,A)$ a thickening. 
There is a connected component of $\Si\setminus\G$ for each boundary component $C_i$ of $\Si$ not contained in $A$, this component is homeomorphic to $C_i\times (0,1]$.
We denote by $\widetilde{\Si}_i$ the compactification of $C_i\times (0,1]$ to $C_i\times I$.
We denote by $\Si_\G$ the surface obtained 
by cutting $\Si$ along $\G$, that is taking the disjoint union of the $\widetilde{\Si}_i$. Let
$$g_\G: \Si_\G \rightarrow \Si$$
be the gluing map.  We denote by $\widetilde \G_i$ the boundary component of the cylinder $\widetilde{\Si}_i$ that comes 
from the graph (that is $g_\G(\widetilde \G_i)\subset \G$) and by $C_i$ the one coming from a boundary component of $\Si$ (that is $g_\G(C_i)\subset \partial \Si$). From now on, we take the convention that $C_i$ is identified with $C_i \times \{1\}$ and that $\TG_i$ is identified with $C_{i} \times \{0\}$.
We set $\Si_i:=g_\G(\widetilde{\Si}_i)$ and $\G_i:=g_\G(\widetilde{\G}_i)$. Finally we denote $g_\G(C_i)$ also by $C_i$ since $g_\G|_{C_i}$ is bijective.
The orientation of $\Si$ induces an orientation on every cylinder $\widetilde{\Si}_i$ and on its boundary components. 
\end{notation}

\section{T\^ete -\`a-t\^ete graphs}
\label{00puretat}

We now consider metric relative ribbon graphs $(\G,A, d)$. 
A metric graph is a graph $\G$ together with lengths $l(e)\in \R$ for every edge $e\in e(\G)$. In an edge $e$, we take a homogeneous
metric that gives $e$ total length $l(e)$. 
We consider the distance $d(x,y)$ on $\G$ given by the minimum of the lengths of the paths joining $x$ and $y$. 

\begin{definition}
A {\em walk} in a graph $\G$ is a continuous mapping 
$$\gamma:I\to \G,$$
from an interval $I$, possibly infinite, and such that for any $t\in I$ there exists a neighbourhood around $t$ where $\gamma$ is injective.
\end{definition}

The notion of safe walk is central in this paper. We start by a purely graph theoretical definition.

\begin{definition}[Safe walk]\label{def:safe}
Let $(\G,A)$ be a metric relative ribbon graph. A safe walk for a point $p$ in the interior of some edge is a walk 
$\gamma_p:\R_{\geq 0} \rightarrow \Gamma$ with $\gamma_p(0)=p$ and such that:

\begin{enumerate}
\item[(1)] The absolute value of the speed $|\gamma_p'|$ measured with the metric of $\G$ is constant and equal to  $1$. Equivalently, the safe walk is parametrized by arc length, 
i.e. for $s$ small enough $d(p,\gamma_p(s))=s$. 
\item[(2)] when $\gamma_p$ gets to a vertex, it continues along the next edge in the given cyclic order.
\item[(3)] If $p$ is in an edge of $A$, the walk $\gamma_p$ starts running in the direction prescribed by the orientation of $A$.
\end{enumerate}
In this paper, we always work with restrictions of safe walks to intervals. An $\ell$-safe walk is the restriction of a safe walk to the interval $[0,\ell]$.
If a length is not specified when referring to a safe walk, we will understand that its length is $\pi$.
\end{definition}

The notion in (2) of \emph{continuing along the next edge in the order} of $e(v)$ is equivalent to the notion of \emph{turning to the right} in every vertex for paths
parallel to $\G$ in any thickening surface $(\Si,\G,A)$ as it is explained in \cite{Camp1}.

Following the first author choice in \cite{Camp1}, in the first part of the paper, we adopt the convention of working mainly with  $\pi$-safe walks. In \Cref{sec:signed} we will need safe walks of different lengths, 
that is why we introduce it in such a generality.

\begin{remark}[Safe walk via cylinder decomposition.]\label{re:safe_cylinder} Taking a thickening $(\Si,\G,A)$, the previous definition extends to safe walks $\gamma_p$ starting also at $p \in v(\G)$ by replacing 
condition \emph{(2)} by  
\begin{enumerate}
 \item[\emph{(2')}] the path $\gamma_p$ admits a lifting $\tilde{\gamma}_p:\R \rightarrow \Si_\G$ in the cylinder 
decomposition of $\Si_\G$ (see  \Cref{not:cut}), which runs in the opposite direction to the one indicated by the orientation induced at the boundary of the cylinder.
\end{enumerate}
\end{remark}

\begin{remark}[Definition of $\gamma_p$ and $\omega_p$.]\label{re:cut}
Given a relative ribbon graph $(\G,A)$ and a thickening $(\Si,\G, A)$, we make some observations in order to help fixing ideas: 
\begin{enumerate}

\item[(a)] every vertex $v\in v(\G)$ has as many preimages by $g_\G$ as its valency $e(v)$. These preimages belong to certain $\widetilde{\G}_i\cont \widetilde{\Si}_i$ for 
certain cylinders $\widetilde{\Si}_i$ which could occasionally be the same. An interior point of an edge not included in $A$ has always two preimages. 
An interior point of an edge included in $A$ has always one preimage.

\item[(b)] for every point $p\in \G$ and every oriented direction from $p$ along $\G$ compatible with the orientation of $A$ there is a safe walk starting on $p$ following 
that direction. This safe walk admits a lifting to one of the cylinders $\widetilde{\Si}_i$. 

In particular:
\begin{enumerate}
\item[(b1)] For $p$ an interior point of an edge not belonging to $A$, that is for $p\in \G\setminus (v(\G)\cup A)$, only $2$ starting directions for a safe walk are 
possible, corresponding to the two different preimages of $p$ by $g_\G$. We will denote the corresponding safe walks by $\gamma_p$ and $\omega_p$. If $p$ is at the
interior of an edge contained in $A$ only one starting direction for a safe walk at $p$ is possible.

\item[(b2)] For a vertex $v$, not belonging to $A$ there are as many starting directions as edges in $e(v)$, and for any vertex $v$ belonging to $A$, there are as many starting 
directions as edges in $e(v)$ minus $1$ (the edge in $A$ whose orientation arrives to $v$ does not count).
\end{enumerate}
\end{enumerate}

\end{remark} 

\begin{definition}[\Tat property and \tat graph] \label{def:tat}
Let $(\G,A,d)$ be a metric relative ribbon graph without univalent vertices. 
We say that $\G$ satisfies the \emph{$\ell$-\tat property}, or that $\G$ is an \emph{$\ell$-\tat graph} if
\begin{itemize}

\item[$\bullet$] For any point $p\in \G\setminus (A\cup v(\G))$ the two different $\ell$-safe walks starting at $p$ (see \Cref{re:cut}), that we denote 
by $\gamma_p$ and $\omega_p$, satisfy $\gamma_p(\ell)=\omega_p(\ell)$.
\item for a point $p$ in $A\setminus v(\G)$, the end point of the unique $\ell$-safe walk starting at $p$ belongs to $A$.

\end{itemize}

If $(\Si,\G,A)$ is the regular thickening of the graph $(\Gamma,A)$ where $A$ denotes the corresponding union of boundary components, 
we say that $(\G,A)$ \emph{gives a relative $\ell$-\tat structure to} $(\Si,\G,A)$ or that $(\G,A)$ is a \emph{relative $\ell$-\tat graph or spine for} $(\Si,\G,A)$.  

If $A=\emptyset$, we call it a \emph{pure} $\ell$-\tat structure or graph.
\end{definition}

\begin{lemma}[Lemma and Definition]
\label{lem:sigma} 
For an $\ell$-\tat graph $(\G,A,d)$, the mapping $\sigma_\G:\G\to \G$ defined by 
$\sigma_\G(p)=\gamma_p(\ell)$ is a well defined homeomorphism.
\end{lemma}
\begin{proof}
Let 
$$\sigma:\coprod_i \widetilde \G_i\to \coprod_i \widetilde \G_i$$
be the homeomorphism which restricts to the metric circle $\widetilde \G_i$ to the negative rotation of amplitude $\ell$ (move each point to a point which is at distance 
$l$ in the negative sense with respect to the orientation induced as boundary of the cylinder). 
The \tat property implies that $\sigma$ is compatible with the gluing $g_\G$ at any point which is not the preimage of a vertex. By continuity
the compatibility extends to all the points. The mapping $\sigma$ descends to the mapping $\sigma_\G$. This proves the assertion.
\end{proof} 

\begin{remark}\label{re:S1_segm} There is a special and easy case for $\pi$-\tat graphs: when $\G$ is homeomorphic to 
$\MS^1$. The thickening surface is in this case the cylinder. 

If $\G$ is $\MS^1$, then the only possibilities for $\sigma_\G$ are the identity or the $\pi$ rotation (for the homogenous metric). Then $\G$ has total length of $2\pi/n$ for some $n \in \N$.
\end{remark}

\begin{corollary}\label{cor:sigma}The homeomorphism $\sigma_\G$ has the following properties:
\begin{enumerate}
\item it is an isometry,
\item it preserves the cyclic orders of $e(v)$ for every $v\in v(\G)$,
\item it takes vertices of valency $k>2$ to vertices of the same valency, 
\item it has finite order.
\end{enumerate}
\end{corollary}

\begin{proof}
Point $(1)$ follows from \Cref{lem:sigma} because $\sigma_\G$ is a homeomorphism that is an isometry 
restricted to the edges.
Point $(2)$ follows from the fact that $\gamma_p(t+t')=\gamma_{\gamma_p(t)}(t')$. More precisely, take $e'$ appearing 
after $e$ in the cyclic order of $e(v)$. Take $p\in e$ and $\gamma_p(t)\in e'$ for $t$ small. Then, it is clear that 
the next edge of $\sigma_\G(e)$ in the cyclic order of  $v(\sigma_\G(v))$ is the edge of $\gamma_p(l+t)$.  

Point $(3)$ is immediate since $\sigma_\G$ is a homeomorphism. 

To see that $\sigma_\G$ has finite order when it is not $\MS^1$, we observe that $\sigma_\G$ induces a permutation between edges and vertices of $\G'$ and is an isometry. Then, it has finite order. When the graph is homeomorphic to $\MS^1$ it is not clear, a priori, that $\sigma_\G$ permutes vertices but the result follows from the observation in \Cref{re:S1_segm}.
\end{proof}

Note that condition (2) implies that $\sigma_\G$ can be extended to a homeomorphism of a thickening surface. 
The induced \tat homeomorphisms that we will define (see  \Cref{def:phi_Gtau} or \ref{def:phi_Gper}) are certain extensions of it.

\begin{corollary}
\label{cor:sigma_fix}
The following assertions hold:
\begin{enumerate}\item If $\sigma_\G|_e=id$  for some edge $e$, then $\sigma_\G$ is the identity. 
\item for every $m\in \N$ the homeomorphism $\sigma_\G^m$ is also induced by a \tat graph,
\item If $\sigma_\G^m|e=id$ for some edge $e$, then $\sigma^m_\G$ is the identity. 
\end{enumerate}
\end{corollary}
\begin{proof} Given $\sigma_\G$ as in $(1)$, since it preserves the cyclic order at every $v$, then it fixes all the edges
adjacent to the vertices of $e$. Since the graph
is connected, this argument extends to the whole graph and the statement follows. 

To see $(2)$ and find a \tat graph for $\sigma_\G^m$ one can take the same combinatorial graph $\G$ with edge lengths equal to the ones of $\G$ divided by $m$.

Assertion $(3)$ follows from $(1)$ and $(2)$. 
\end{proof}

\begin{example}\label{ex:Kpq2}
Given a ribbon graph $K_{p,q}$ as in  \Cref{ex:Kpq}, we endow it with a metric that gives length $\pi/2$ to any edge. 
 Then we have a \tat graph. Moreover, the homeomorphism $\sigma_\G$ has order $\lcm(p,q)$. There are two special orbits, 
 one of $p$ vertices and another one of $q$ vertices.


\end{example}

When dealing with \tat graphs and $\sigma_\G$, we can assume that we have
particularly simple combinatorics as the following lemma shows:

\begin{lemma}
\label{lem:assum} 
If $(\G,A, d)$ is a \tat graph, only modifying the underlying combinatorics (without changing the topological type of $\G)$, we can ensure we are in one of the following cases: 
\begin{enumerate}
\item unless $\G$ is either homeomorphic to $\MS^1$ or contractible, all the vertices have valency $\geq 3$,
\item there are no edges joining a vertex with itself and there is at most one edge joining two vertices. In this case the restriction $\sigma_\G|_{v(\G)}$ 
determines $\sigma_\G$.
\item all the edges have the same length,
\item the graph satisfies properties in (2) and (3) simultaneously. 
\end{enumerate} 
\end{lemma}
\begin{proof}
If $\G$ is either homeomorphic to $\MS^1$ or contractible, after \Cref{re:S1_segm}, the proof is trivial. 
 
Let's see the case where $\G$ is neither homeomorphic to $\MS^1$ nor contractible. 
To get a graph as in $(1)$ we can consider the graph $\G'$ forgetting the valency-$2$-vertices of $\G$ but keeping distances. It is clearly a \tat graph. 

To get a graph as in $(2)$ we consider the graph $\G$ as in $(1)$. We add as vertices some mid-point-edges $q_1$,...,$q_m$ to $\G'$ in order the new graph  has no loops 
and no more than one edge between any pair of vertices. Now, we have to add as vertices any other point $p$ for which $\gamma_p(\pi)$, the end of the safe walk for $\G$,
is one of these new vertices $q_i$. Since $\sigma_{\G}$ takes isometrically edges to edges, it will take midpoint edges to midpoint edges of $\G$. Then we have to add at
the most all the mid point edges of $\G$ as vertices to reach the desired graph.

Moreover, we note that in a graph as in $(2)$, the image $\sigma_\G(e)$ of an edge $e$ joining $v_i$ and $v_j$, has to be the only edge joining $\sigma_\G(v_i)$ and 
$\sigma_\G(v_j)$. Then, $\sigma_\G|_{v(\G)}$ determines $\sigma_\G$. 

To find a \tat graph as in $(3)$, we start with a graph  $\G$ as in $(1)$. 
The homeomorphism $\sigma_\G$ permutes edges. Moreover the \tat condition says that certain summations of the lengths $\{l(e)\}_{e\in e(\G)}$ are equal to $\pi$. 
We consider only the summations that come from measuring the lengths of the safe walks that start in vertices of $\G$, which are a finite 
number. We collect all these linear equations in the variables $l(e)$ in a system $S$. We consider the system of equations $S'$ by replacing the independent term $\pi$ in 
the equations of $S$ by $1$. 
It is clear that there exist positive rational solutions $l(e)$ of the system $S'$. Let $N$ be a common denominator. We consider the graph $\G'$ by subdividing every edge 
$e$ of $\G$ into $N\cdot l(e)$ edges of length $\pi/N$ obtaining the desired graph. 

If $\G'$ does not satisfy properties in $(2)$, replacing $N$ by $2N$ it does. You can also add the middle points of all the edges as vertices and finish as in the proof of $(2)$. 
\end{proof}

\begin{remark}This lemma shows that, without loss of generality, we can assume that a \tat graph has all the edges of the same length 1 and safe walks of integer length $\ell$. Anyway, we keep the original definition by historical reasons.
\end{remark}

A way to obtain relative \tat graphs from pure ones is the notion of $\epsilon$-blow up.

\begin{definition}[\textbf{$\epsilon$-Blow up of $(\Sigma,\G)$ at a vertex of $\G$}]
\label{def:blow}
Let $\G$ be a pure $\ell$-\tat graph and $(\Sigma,\G)$ be a thickening surface. Let $v$ be a vertex of valency $p$. We consider the oriented real blow up of $\Si$ at $v$. 
We denote by $\Si'$ and $\G'$ the transformations of $\Si$ and $\G$. Note that $\Si$ has one more boundary component and $\G'$ has changed the vertex $v$ by a circle 
$A\cong \MS^1$ with $p$ edges attached. Away from $v$, we consider the metric  in $\G'$ as in $\G$. We assign the length  $2\epsilon$ to the new edges in $\G'$ along 
$\MS^1$ and redefine the length of the edges corresponding to each $e \in e(v)$ by $length(e)-\epsilon$. 
(See \Cref{fig:blowup} below). We do this at every vertex on the orbit of $v$ by the $\sigma_\G$ and denote the resulting space by $Bl_v(\G, \epsilon)$. We say that it is the result of performing the $\epsilon$-blowing up of $\G$ at $v$.
\end{definition}

It is immediate to check that $(\G',A)$ is a relative \tat graph and that $(\Si',\G',A)$ is a thickening. 

\begin{figure}[!ht]
\centering
\includegraphics[scale=0.5]{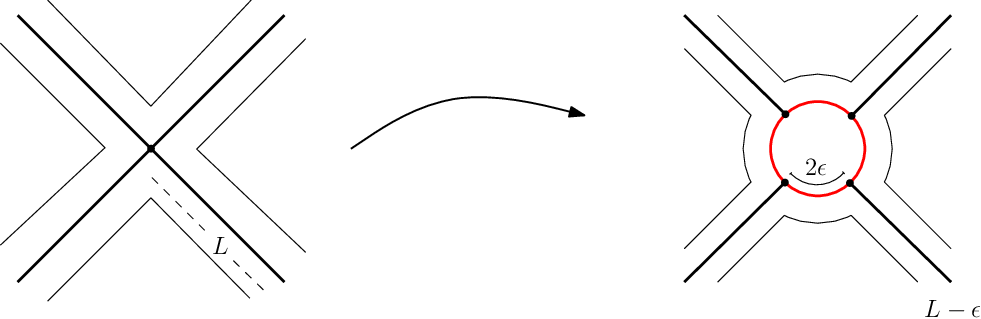}
\caption{Blow-up some vertex $v$ of valency $4$.}
\label{fig:blowup}
\end{figure}


\section{Periodic homeomorphisms up to isotopy}
\label{sec:iso}

In this section, for the reader convenience, we recall some elementary known facts about periodic homeomorphisms both in $MCG(\Si)$ and in
$MCG(\Si, \partial \Si)$. We fix notation and conventions for the rest of the work.

\begin{notation}\label{not:iso}
Let $\Si$ be a fixed surface with non-empty boundary.

We denote by $[\phi]$ the class of homeomorphisms of $\phi$ up to isotopy. The group, with the composition, of these equivalence classes is called the Mapping Class Group $MCG(\Sigma)$.  Two homeomorphisms of the same class are said to be \emph{boundary-free isotopic} or simply \emph{isotopic}.

A homeomorphism of a surface is \emph{periodic in} $MCG(\Sigma)$ or \emph{periodic up to boundary-free isotopy} if there exists $n\in \mathbb{N}$ such that $[\phi^n]=[id]$.

Let $B$ be a subset of $\partial\Sigma$.
We denote by  $[\phi]_{B, \phi|_B}$ the class of homeomorphisms that are isotopic to $\phi$ by an isotopy that coincides with $\phi|_B$ at $B$ all along the family.  We denote by $MCG(\Si, B, \phi|_B)$ the set of these classes. If $\phi$ and $\psi$ are in the same class we say they are \emph{isotopic relative to the action $\phi|_B$}. 
If the action is the identity on $B$, we omit the action in the notation and recover the classical notion of \emph{isotopy relative to} $B$, that means that all the homeomorphisms in the isotopy fix $B$ pointwise. We write these classes simply by $[\phi]_B$. 
We always consider the case in which $B$ is a union of connected components of $\partial\Sigma$ and denote it by $B=\partial^1\Si$. In the case $B=\partial\Si$ we will simply write $[\phi]_{\partial,\phi|_{\partial}}$ or $[\phi]_\partial$. Note that in this last case we recover the classical notion of mapping classes fixing pointwise the boundary.

We denote by $MCG^+(\Si)$,  $MCG^+(\Si, B, \phi|_B)$ or $MCG^+(\Si, B)$ the corresponding restrictions to homeomorphisms preserving orientation. 
\end{notation}

\begin{remark}
Observe that $MCG(\Si, \partial, \phi|_{\partial})$ is not a group. However, the group $MCG(\Si, \partial)$ acts  transitive and freely on it.
\end{remark}

\subsection{Periodic mapping classes in $MCG(\Si)$}

In this subsection we assume all isotopies are boundary-free. We focus on periodic elements of $MCG^+(\Si)$. 

A key result, only true in dimension 2 (see \cite{RayScott}), is the following classical theorem: 

\begin{theorem}[Nielsen's Realization Theorem \cite{Niel}, also see Theorem 7.1 in \cite{Farb}]\label{theo:niel} If $\phi^n$ is isotopic to the identity, then there exists $\hat{\phi}\in [\phi]$ such that 
$\hat{\phi}^n=Id$. Moreover, there exists a metric on $\Si$ such that $\hat{\phi}$ is an isometry.
\end{theorem}

We will use the following well-known fact:

\begin{lemma}\label{lem:fixed}
Let $\phi:\Si \rightarrow \Si$ be an orientation-preserving isometry of $\Si$. Then either the fixed points are isolated and disjoint from the boundary or $\phi$ is the identity. Moreover, if $\phi$ is 
a periodic homeomorphism, then the points with non trivial isotropy for the action generated by $\phi$ are also isolated and disjoint from the boundary. 
As a consequence, given a periodic homeomorphism $\phi$ of a surface that leaves all the boundary components
invariant, the restriction to any boundary component has the 
same order than $\phi$.
\end{lemma}

\begin{proof}
The second paragraph of the proof of Theorem 6.8 in \cite{Farb} (page 202), proves that fixed-points of isometries are isolated. That points with non trivial isotropy for the action generated by $\phi$ are isolated follows from the fact that they are fixed points for some power $\phi^m \neq id$.
	
	The last part of the statement follows by observing that removing points with non-trivial isotropy does not disconnect the surface. 
\end{proof}

\begin{notation}
\label{not:quot} Let $\phi$ be a periodic orientation preserving homeomorphism of $\Si$. We denote by $\Si^\phi$ the orbit space (which is a surface) and 
by $$p:\Si\to \Si^\phi$$ the quotient mapping. The mapping $p$ is a Galois ramified covering map. The set of points in $\Si$ whose orbit has cardinality strictly 
smaller than the order of $\phi$ are called \emph{ramification} points. Its images by $p$ are called \emph{branching} points. 

\end{notation}

\begin{remark}\label{galois}
Since the covering map $p$ is Galois, any point at the preimge by $p$ of a branching point is a ramification point. 
\end{remark}

\begin{definition}\label{def:alex}
Let $\phi$ be a periodic homeomorphism of $\Si$ that leaves a boundary component $C_i\cont\partial \Si$ invariant. 
We cap this boundary component $C_i$ with a disk $D^2$ obtaining a new surface $\Si'$. We extend $\phi$ to a periodic orientation-preserving homeomorphism of $\Si'$ as
follows: if $\theta$ is the angular and $r$ the radial coordinates for $D^2$ then we define $\Phi:D^2 \rightarrow D^2, (\theta,r)\mapsto (r,\phi(\theta))$. 
The homeomorphisms $\Phi$ and $\phi$ glue along $C_i$.
We call this extension procedure  the \emph{Alexander trick}.
\end{definition}

\begin{remark}\label{re:iso_graph} Recall that given $\phi$ and $\psi$  two homeomorphisms of $\Si$ that both leave a spine $\G$ invariant, if $\phi|_\G$ and $\psi|_\G$ are
isotopic, then $\phi$ and $\psi$ are isotopic. In other words, the isotopy type of the restriction of a homeomorphisms to an invariant spine determines the isotopy type of 
the homeomorphism of $\Si$.
\end{remark}

\begin{lemma}\label{lem:invariant} Let $\Si$ be a surface with $\partial \Si\neq \emptyset$ which is not a disk or a cylinder.
Let $\phi:\Si \rightarrow \Si$ be an orientation preserving homeomorphism.  
Then $\phi$ is periodic up to isotopy if and only if there exists $\hat{\phi}\in [\phi]$ such that there exists a spine $\G$ of $\Si$ which is invariant by $\hat{\phi}$.
\end{lemma}

\begin{proof}
Assume $\phi$ is periodic up to isotopy. By Nielsen's realization Theorem we can assume that $\phi$ is periodic. 
Let $\Si^{\phi}$ be the orbit space of $\phi$. 

The quotient map $p:\Si \rightarrow \Si^{\phi}$ 
is a branched covering map whose ramification points are isolated and are contained in the interior of $\Si$ by 
\Cref{lem:fixed}. 
Pick any spine $\G^{\phi}$ for $\Si^{\phi}$ containing all the branch points.
The set $\Gamma := p^{-1}(\Gamma^\phi)$ is invariant and it is a spine for being the pre-image of a spine containing the branching locus.

Conversely, assume that an invariant spine $\G$ exists for some $\hat{\phi}\in [\phi]$. We consider the spine  only with vertices of valency greater than $2$. Since $\hat{\phi}$ leaves the spine invariant, then $\hat{\phi}$ acts as a permutation on edges and vertices. Then, there is a power of $\hat{\phi}$, say $\hat{\phi}^m$ that 
leaves all the edges and vertices invariant.
Thus, $\hat{\phi}^m|_\Gamma$ is isotopic to the identity, and hence $\hat{\phi}^m$ too.
\end{proof}

\begin{remark}In the theorem above we excluded the cases when $\Si$ is a cylinder or a disk for being trivial. In these cases every homeomorphism is isotopic to a periodic homeomorphism.
\end{remark}

Not every spine obtained in the proof of the previous lemma accepts a \tat structure such that $\sigma_\G = \phi|_\G$ (see \Cref{ex:counterexample0}). In \Cref{thm:000} we will see show how to find
one that accepts it.

\begin{notation}
\label{not:cut_aut}
If a homeomorphism $\phi$ of $\Si$ leaves a spine $\G$ invariant, then the homeomorphism lifts to a homeomorphism of $\Si_\G$ that we denote by $\widetilde{\phi}$. 
\end{notation}

If $\phi$ is a periodic homeomorphism of $\Si$ that leaves a boundary component $C_i$ invariant, then $\phi|_{C_i}$ is a 
periodic homeomorphism of $\MS^1$. Then, we can consider the usual Poincare's rotation number $rot(\phi|_{C_i})$ 
with values in $(0,1]$.

It is well known that rotation numbers classify periodic homeomorphisms of the circle up to conjugation. As an easy 
consequence we have the following:
\begin{remark}\label{re:cyl} 
Any orientation preserving periodic homeomorphism of the cylinder which leaves invariant each boundary component is conjugate to a rotation.
\end{remark}
%
%


\subsection{Elements of $MCG(\Si, \partial^1 \Si)$ which are periodic in $MCG(\Si)$}

Now we take a look at the Mapping Class Group where homeomorphisms and isotopies fix pointwise the union of some boundary
components, that we denote by $\partial^1\Si\subset\partial\Si$ (recall  \Cref{not:iso}).

Consider a non-empty union $\partial^1\Si$ of the boundary components. We study the elements of 
$[\phi]_{\partial^1}\in MCG^+(\Si, \partial^1\Si)$ that are boundary-free isotopic to a periodic one.

If $\Si$ is the disk, it is clear that $MCG^+(\Si, \partial\Si)\simeq 0$. If $\Si$ is the cylinder, then $MCG^+(\Si, \partial\Si)\simeq \Z$ and it is generated by 
the right (or left) Dehn twist along a curve that is parallel to the boundary components. All its elements are boundary-free isotopic to the identity.  

\begin{remark}\label{rem:convention}
In this work we take the convention that \emph{negative} Dehn twists are \emph{right-handed} Dehn twists. See \Cref{fig:right_twist}. 
\end{remark}

\begin{figure}[!ht]
\centering
\includegraphics[scale=0.5]{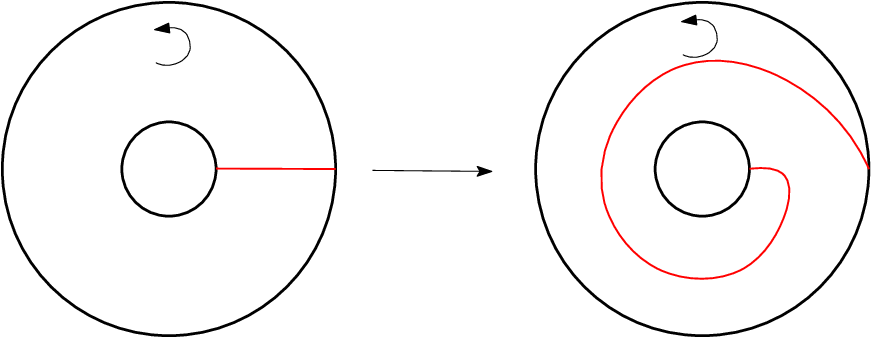}
\caption{We see an oriented annulus on the left and the image of the red curve by a right-handed Dehn twist on the right. The little curve on the top of each picture represents the orientation on the annulus.}
\label{fig:right_twist}
\end{figure}

Let $\phi$ be an orientation preserving homeomorphism of $\Si$, fixing pointwise $\partial^1\Si$ and boundary-free isotopic to
a periodic one $\hat{\phi}$.

We need the notion of fractional Dehn twist coefficients. 
The fractional Dehn twist coefficient at a boundary component can be understood as the  difference between $\phi$ and a
truly periodic representative at that boundary component. We define fractional rotation number in a slightly different way than the usual one, since it is more convenient for our applications, and allows a self contained exposition for the case that we need.

We start recalling some facts about Dehn twists. If we do not say the contrary, the letter $D$ with a subindex, denotes a negative (right-handed) Dehn twist along some curve that will be clear from the context.

\begin{lemma}\label{lem:freedehn}
Let $\Si$ be a surface with $r>0$ boundary components that is not a disk or an cylinder. Then the group generated by the Dehn twists $D_1, \ldots D_r$ along curves parallel to each boundary component is free 
abelian of rank $r$. 
\end{lemma}
\begin{proof}
Lemma 3.17 in \cite{Farb}.
\end{proof}

As a consequence, we find the following result.

\begin{lemma}\label{lem:idehn}
Let $\Si$ be a surface that is neither a disk nor a cylinder. Let $\partial^1 \Si$ be non-empty union of $r$  boundary components of $\Sigma$. 
Let $\phi$ be an orientation preserving homeomorphism of $\Si$ fixing $\partial^1 \Si$ pointwise. If $\phi$ is boundary-free isotopic to the identity then there exist 
unique integers $n_1, \ldots, n_r$ such that we have the equality
$[\phi]_{\partial^1}=[D_1^{n_1}]_{\partial^1}\cdot\ldots \cdot[D_r^{n_r}]_{\partial^1}$.
\end{lemma}



\begin{definition}[fractional Dehn twist coefficient]\label{def:gen_rot}
Let $\Si$ be a surface that is neither a disk nor a cylinder. Let $\partial^1 \Si$ be non-empty union of $r$  boundary components of $\Sigma$. 
Let $\phi: \Si \rightarrow \Si$ be a homeomorphism fixing pointwise $\partial^1\Si$ and boundary-free isotopic to a periodic one. 
Let $m\in\N$ such that $[\phi^{m}]=[id]$. Let $t_1, \ldots, t_r$ be integers such that 
$[\phi^{m}]_{\partial^1}=[D_1^{t_1}]_{\partial^1}\cdot\ldots\cdot [D_r^{t_r}]_{\partial^1}$. We define the 
\emph{fractional Dehn twist coefficient} at $C_i$ by 
$$rot_{\partial^1}(\phi,C_i):=t_i/m.$$
\end{definition}

Note that the  fractional Dehn twist coefficients do not depend on the number $m$ we choose to compute them or the representative $\phi\in[\phi]_{\partial^1}$.

Now we describe how to compute the fractional Dehn twist coefficient in terms of an invariant spine.
First of all, we observe that in the isotopy class fixing $\partial^1\Si$ pointwise there is always a representative 
fixing a spine. Indeed, the following lemma is elementary:
\begin{lemma}
\label{lem:collar}
Let $\phi$ be an orientation preserving homeomorphism of a surface $\Si$ which fixes pointwise a non-empty union $\partial^1\Si$ of boundary components, and which is
boundary-free isotopic to a periodic homeomorphism $\hat{\phi}$ . Then there
there exists a collar $U$ of $\partial^1\Si$, a homeomorphism $f:\Si\to \overline{\Si\setminus U}$, and a homeomorphism $\psi$ of $\Si$ such that: $\psi$ is 
isotopic relative to $\partial^1\Si$ to $\phi$ and the restriction
$\psi|_{\overline{\Si\setminus U}}$ is periodic and equal to $f\comp \hat{\phi}\comp f^{-1}$.

In particular $\psi$ leaves a spine $\G$ invariant and $\psi|_{\G}$ is periodic. 
\end{lemma}

%

Let $\Si$ be an oriented surface with non-empty boundary that is neither a disk nor a cylinder. 
Let $\phi$ be an orientation preserving homeomorphism of $\Si$ that fixes a non-empty union $\partial^1\Si$ of components of the boundary and which is 
boundary-free isotopic to a periodic one. Let $A$ be the union of the remaining components of the boundary. Suppose that there exists a relative spine $(\G,A)$ in $(\Si,(\G,A))$ 
which is invariant by $\phi$.

We cut $\Si$ along $\G$ into a disjoint union of cylinders, one for each component $C_i$ of $\partial^1\Si$. We use \Cref{not:cut,not:cut_aut}. 
We lift the retraction $\Si\to \G$  to a retraction $\Si_\G\to\widetilde{\G}$ and the homeomorphisms $\phi$ to a homeomorphism $\widetilde{\phi}$ of $\Si_\G$. 
Let $\frac{p_i}{n}$ be the rotation number of $\widetilde{\phi}|_{\widetilde{\G}_i}$. Choose in the cylinder $\widetilde{\Si}_i$ a retraction line $L_i$ from 
$C_i$ to $\widetilde{\G}_i$. Consider the orientation in $\widetilde{\Si}_i$ inherited from the orientation in $\Sigma$.
We take the classes $[L_i]$ and $[\phi(L_i)]$ in  $H_1(\widetilde{\Si}_i, \partial \widetilde{\Si}_i)$. The class 
$$\widetilde{\phi}|_{\widetilde{\Sigma}_i}^n([L_i])-[L_i]$$ 
belongs to  $H_1(\widetilde{\Si}_i)$ since $\widetilde{\phi}|_{\widetilde{\Sigma}_i}^n$ is the identity at the boundary.
Let 
$$k_i:=(\widetilde{\phi}|_{\widetilde{\Sigma}_i}^n([L_i])-[L_i])\cdot [L_i],$$ 
that is the oriented intersection number of the two homology classes.

\begin{lemma}
\label{lem:equal_rot}
We have the equality $rot_{\partial^1}(\phi,C_i)= k_i/n$. In particular $k_i$ does not depend on the chosen spine $\G$ or even on the
representative of $[\phi]_{\partial^1}$.
\end{lemma}
\begin{proof}
Note that $\phi^n$ fixes $\G$ and that the lifting 
$\widetilde{\phi}^n|_{\widetilde{\Si}_i}$ is isotopic relative to the boundary to the composition of $k_i$ right boundary Dehn twist if $k_i$ is positive (and $-k_i$ left Dehn 
twists if $k_i$ is negative) around the boundary component $C_i$. Then $[\phi^n]_{\partial^1}=[D_1]^{k_1}\cdot...\cdot [D_r]^{k_r}$ and the result follows. 
\end{proof}

\begin{corollary}
\label{cor:iso_rot}
Let $g,h: \Si \rightarrow \Si$ be two homeomorphisms that fix pointwise a non empty union $\partial^1\Si$ of components of the boundary $\partial\Si$. Let $A$ be the union
of the remaining boundary components.
Assume that both preserve a common relative spine $(\G,A)$ and that they coincide
and are periodic at it. Then the equality $rot_{\partial}(g,C_i)=rot_{\partial}(h,C_i)$ holds for every $i$ if and only if $h$ and $g$ are isotopic relative to 
$\partial^1\Si$.
\end{corollary}

\begin{corollary}
\label{cor:rot_red} 
Let $\phi:\Si\to \Si$ be a homeomorphism that fixes pointwise a non-empty union $\partial^1\Si$ of components of the boundary, and that is isotopic to a periodic 
homeomorphism $\hat{\phi}$. Let $C_i$ be a component in $\partial^1\Si$. Then the usual rotation number up to an integer $rot(\hat{\phi}|_{C_i})$ equals 
$|rot_{\partial^1}(\phi, C_i)-\left\lfloor  rot_{\partial^1}(\phi, C_i)\right\rfloor|$ where $\left\lfloor x\right\rfloor$ is the biggest integer less that $x$. 
\end{corollary}

\begin{remark}
We observe that by our convention on \Cref{rem:convention}, negative (or equivalently right-handed Dehn twists) produce positive fractional Dehn twist coefficients. This is in accordance with the previous notions of fractional Dehn twist coefficient in the literature.
\end{remark}

\section{\Tat twists and periodic classes of $MCG(\Sigma)$.}
\label{sec:tatatat}

In this section we introduce \tat twists, which are mapping classes associated with \tat graphs, and provide a wide generalization of Dehn twits. In fact we show that all periodic classes of $MCG(\Sigma)$ that leave invariant at least one boundary component is a \tat twist in two different ways:
\begin{itemize}
\item to a \tat graph we associate an element of $MCG( \Sigma, \partial^1\Sigma)$, that we call the \tat twist associated with the \tat graph (see \Cref{def:phi_Gtau}) in \Cref{sec:signed}, and prove that all elements of $MCG( \Sigma, \partial\Sigma) $ periodic in $MCG(\Si)$, with positive fractional Dehn twits coefficient are \tat twists. In fact enlarging the definition of \tat graph to that of the signed \tat graphs (see \Cref{def:tat_tau}), we can represent all elements of $MCG( \Sigma, \partial\Sigma) $ periodic in $MCG(\Si)$, both with positive and negative fractional Dehn twist coefficient. The signed \tat graphs and twists were also introduced by C. Graf in \cite{Graf1}.

\item to a \tat graph we associate a (really) periodic homeomorphisms that we call periodic \tat twist associeted with the graph (see \Cref{def:phi_Gper}) in \Cref{sub:comp_per}, and show that any periodic homeomorphism is boundary free isotopic to a periodic \tat twist.
\end{itemize} 


\subsection{Signed \tat graphs and representatives in $MCG(\Si, \partial^1\Si)$.}\label{sec:signed}
Our objective in this section is to represent homeomorphisms which fix pointwise a union $\partial^1\Sigma$ of components of the boundary, and that are boundary-free isotopic to a periodic one.

We start with the definition of signed \tat graph that generalizes \Cref{def:tat}. 

We will work directly for the class of relative graphs in order to avoid repetitions. The main results are the \Cref{thm:000,thm:rel}.



Let $(\G,A)$ be a metric relative ribbon graph. Let $(\Si,\G,A)$ be a thickening, let  
$$g_\G:\Si_\G\to \Si$$
be the gluing  mapping as in \Cref{not:cut}. 

We start making an extension of \Cref{re:cut} adding point (b'): 
\begin{remark}[Definition of $\gamma_p^-$, $\omega_p^-$, $\gamma_p^0$,  $\gamma_p^+$ and $\omega_p^+$]\label{re:cut3}
$\ $

\begin{itemize}
\item[(b')] for every point $p\in \G\setminus v(G)$ and every of the two possible oriented direction from $p$ along $\G$, there is a walk starting on $p$ following each of this directions, such that in every vertex $v$, the walk continues along the \emph{previous} edge in the cyclic order of $e(v)$. We denote by $\gamma^-_p$ and $\omega^-_p$ these walks of length $\pi$ and speed 1.  
Each of the oriented directions at $p$ corresponds to a point in $q\in g_\G^{-1}(p)$, which lives in a cylinder $\widetilde{\Si}_i$. These walks  are the image of the negative sense parametrization of the boundary of $\widetilde{\Si}_i$ starting at $q$.
\end{itemize}

We denote by $\gamma^+_p$ and $\omega^+_p$ the usual safe walks of \Cref{def:safe} or \Cref{re:cut} of length $\pi$ and speed 1. 
  We call $\gamma^+_p$ and $\omega^+_p$ the \emph{positive safe walks} and $\gamma^-_p$ and $\omega^-_p$ the \emph{negative safe walks}. 

In the case of points in $A$, since $A$ is oriented, we have also a positive and negative sense for a parametrization. Then, for $p\in A$,
we define $\gamma^+_p$ (respectively $\gamma^-_p$) as the parametrization from $p$ that starts along $A$ in the positive (respectively negative) sense 
and that when reaching a vertex $v$ takes the next (respectively previous) edge in the order of $e(v)$). 

We also define a \emph{safe constant walk} $\gamma_p^0:=p$.

\end{remark}

Before stating the next definition, recall that there is a bijection between the set $\calC$ of boundary components of $\Si$ not in $\calA$ and the cylinders $\widetilde{\Si}_i$'s.
Given a ``sign'' mapping $\iota:\calC\to\{0,+,-\}$, we denote by $\iota(i)$ the image by $\iota$ of the component that corresponds to $\widetilde{\Si}_i$ under the 
bijection.

\begin{definition}[Signed \tat property and graph]\label{def:tat_tau}
Let $(\G,A)$ be a metric relative ribbon graph and let $(\Si,\G,A)$ be a thickening. 
Let $\calC$ denote the set of boundary components of $\Sigma$ which do not 
belong to $A$. Fix a mapping 
$$\iota:\calC\to\{0,+,-\}.$$ 

We say that $(\G,A)$ satisfies the \emph{signed \tat property for} $\iota$ or that $(\G,A,\iota)$ is a 
\emph{signed relative \tat graph} if given any point $p$
contained at the interior of an edge the following properties are satisfied:
\begin{enumerate}
\item if $p$ does not belong to $A$ and $p\in g_\G(\widetilde{\Si}_i)\cap g_\G(\widetilde{\Si}_j)$ for some $i$, $j$, 
then we have the equality 
$$\gamma_p^{\iota(i)}(\pi)=\omega_p^{\iota(j)}(\pi).$$
\item if $p$ belongs to $A$ and $p$ belongs to $g_\G(\widetilde{\Si}_i)$, then the end point $\gamma_p^{\iota(i)}(\pi)$ of the unique signed safe 
walk starting at $p$ belongs to $A$.
\end{enumerate}
\end{definition}

\begin{notation}[Remark and Notation]
Observe that the mapping $\G\setminus v(\G)\to \G$ that sends $p\in \G\setminus v(\G)$ to  $\gamma_p^{\iota(i)}(\pi)$ extends to $v(\G)$ and define a homeomorphism of $\G$ 
that we denote by $\sigma_{(\G,\iota)}$. The proof is as the one of \Cref{lem:sigma}. 
\end{notation}

\begin{definition}[Definition of (signed) \tat twist $\phi_{(\G,A,\iota)}$]\label{def:phi_Gtau}
Let $(\G,A,\iota)$ be a signed relative \tat graph. For every thickening $(\Si, \G,A)$ and for every choice of product structure 
$\widetilde{\Si}_i\approx\widetilde{\G}_i\times I$ we consider the 
homeomorphism 
\begin{equation}\label{eq:psi}\psi_{i}:\widetilde{\G}_i\times I\longrightarrow\widetilde{\G}_i\times I\end{equation}
$$(p,s)\mapsto (\widetilde{\gamma}^{\iota(i)}_p(s\cdot \pi),s)$$
where $\widetilde{\gamma}_p^{\iota(i)}$ is the lifting of the safe walk to $\widetilde{\G}_i$.  
The homeomorphism $\psi_{i}$ of the cylinder can be visualized very easily using the universal covering of the cylinder 
as in \Cref{fig:fixed1}. 

These homeomorphisms glue well due to the properties of the signed relative \tat graph $(\G, \iota)$ and define a
homeomorphism of $(\Si, \G, A)$ that leaves $\partial^1\Si$ fixed pointwise
and $A$ invariant. We denote by $\phi_{(\Si,\G,A,\iota)}$ the resulting homeomorphism of $(\Si,A)$. We call it the induced
\textit{(signed) \tat twist}.

\end{definition}

\begin{remark} 
Given two different product structures of the cylinders, the induced \tat  twists are conjugate by a 
homeomorphism that fixes $\G$.

For two different embeddings in 
$\Si$, the  \tat twists are conjugate by the same homeomorphism of $\Sigma$ that relates the two embeddings.
\end{remark}

\begin{figure}[!ht]
\centering
\includegraphics[scale=0.5]{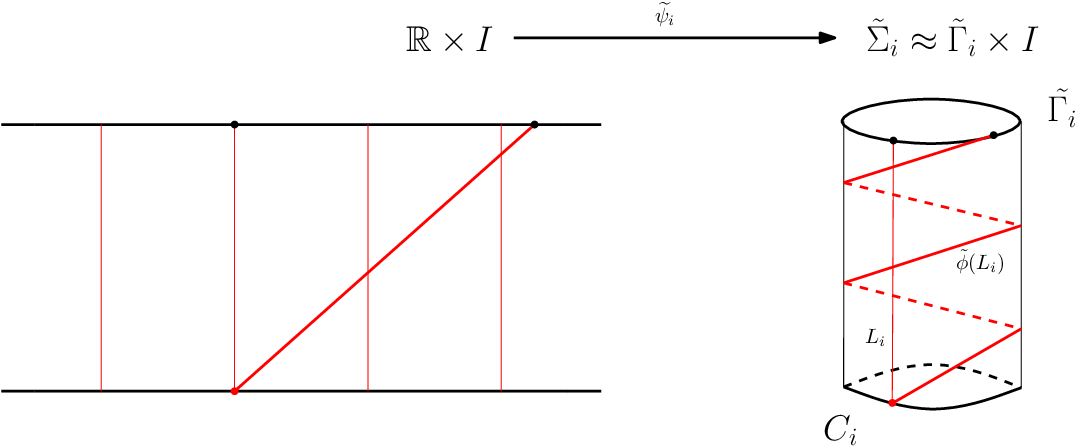}
\caption{Visualization of the lifting $\widetilde{\psi}_i$ of $\psi_i$ to the universal covering of $\widetilde{\Si}_i$. You can easily read (\ref{eq:R_i}) from it. }
\label{fig:fixed1}
\end{figure}

The homeomorphism $\phi_{(\Si,\G,A,\iota)}$ leaves $\G$ invariant and by \Cref{lem:equal_rot} has obviously the following fractional Dehn twist coefficients (see  \Cref{fig:fixed1}): 
\begin{equation}\label{eq:R_i}rot_{\partial^1}(\phi_{(\Si,\G,A,\iota)},C_i)=\iota(i)\cdot\frac{\pi}{length (\widetilde{\G}_i)}.
\end{equation} 

\begin{remark}\label{rem:triv}
Observe that if $\iota(i)=0$ for some $i$, then the homeomorphism $\psi_i$ is the identity and $rot_{\partial^1}(\phi_{(\G,A,\iota)},C_i)=0$. In particular, by \Cref{cor:sigma_fix} we have that the signed \tat twist  must be the identity restricted to $\G$ and hence $rot_{\partial^1}(\phi_{(\Sigma,\G,A,\iota)},C_j) \in \Z$ for all boundary components. So $\phi_{\Si,\G,A,i}$ equals a composition of Dehn twists around curves parallel to the boundary.
\end{remark}

\begin{remark}

If all the signs are positive this notion coincides with the first author's original notion of \tat twists.
\end{remark}
\begin{example}\label{ex:Kpq3}
The homeomorphisms induced by the \tat structures on the $K_{p,q}$ ribbon graphs given in \Cref{ex:Kpq2} are the monodromies fixing the boundary of the Milnor fibrations associated to the 
singularities $x^p-y^q=0$. 
\end{example}
Before going on with the section we note that not every invariant spine admits a metric modeling the corresponding homeomorphism. Indeed, we have the following example.

\begin{example}\label{ex:counterexample0}
In \Cref{fig:counterexample0} you can find a spine $\G$ for the genus $1$ surface with one boundary components that is 
invariant by the $\pi$-rotation along the vertical axes isotoped to be the identity on both boundary components. In particular, it has fractional Dehn twist coefficient equal to $1/2$ on both boundary components.

This ribbon graph does not admit a \tat structure that models the given homeomorphism. Let $L_1$ and $L_2$ be the lengths of
the edges of the graph that meet both cylinders of $\Si_\G$. And let $L_3$ be the other one.

 By \ref{eq:R_i}, since fractional Dehn twist coefficients are the same on both boundary components, we have 
that $1/2 (L_1+L_2)=\pi$ and that $1/2 (L_1+L_2+2L_3)=\pi$. This is turn would imply the equality $L_3=0$ which is not possible since all lengths of edges must be strictly positive.

\begin{figure}[!ht]
\centering
\includegraphics[scale=0.4]{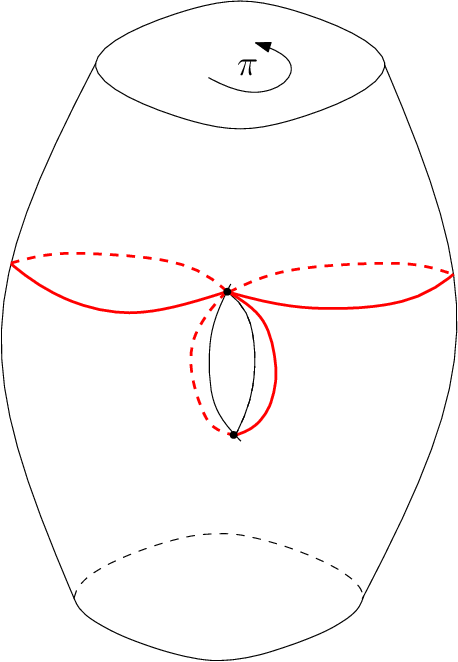}
\caption{A spine for the genus $1$ surface with one boundary components that is invariant by the $\pi$-rotation along the vertical axes.}
\label{fig:counterexample0}
\end{figure}
\end{example}

In the next theorem we see that every  homeomorphism that fixes the boundary pointwise that is boundary-free isotopic to a periodic one, is a 
signed \tat homeomorphism up to isotopy fixing the boundary. We state first
the non-relative case, which improves  \cite[Theorem 3.1.1]{Graf1} as we comment on  \Cref{rem:improvement_graf}.
\begin{theorem}
\label{thm:000}
Let $\Si$ be an oriented surface with non-empty boundary.
Let ${\phi}$ be an orientation preserving homeomorphism fixing pointwise the boundary, and boundary-free isotopic to a periodic one $\hat{\phi}$. 
Then, 
\begin{enumerate}[(i)]
\item there exists a signed \tat spine $(\G,\iota)$ embedded in  $\Sigma$ that is invariant by $\phi$ such that the restriction of $\phi$ to $\G$ coincides with 
$\phi_{(\Sigma,\G,\iota)}$.
\item The isotopy classes relative to the boundary $[\phi]_{\partial}$ and $[\phi_{(\Sigma,\G,\iota)}]_{\partial}$ coincide. 
\item the homeomorphisms $\phi$ and $\phi_{\G, \iota}$ are conjugate by a homeomorphism that fixes the boundary pointwise, fixes $\G$ and is isotopic to the identity in $MCG(\Si, \partial \Si)$.
\end{enumerate}
\end{theorem}

\begin{corollary}\label{cor:original_def}
This theorem characterises the originally defined by the first author \tat twists as (i.e. non-signed \tat twists): orientation preserving homeomorphism fixing pointwise $\partial^1\Si$ and boundary-free isotopic to a periodic one $\hat{\phi}$ with strictly positive fractional Dehn twist coefficients.
\end{corollary}

\begin{remark}\label{rem:improvement_graf}
We compare the above corollary with \cite[Theorem 3.1.1]{Graf1}. In the cited reference, the author proves a similar result but he enlarges the set of {\em permitted} \tat graphs, either by allowing vertices of valency $1$ or by allowing safe walks of different lengths for each boundary component. A homeomorphism leaving a spine invariant leaves the cylindrical decomposition invariant. Our notion imposes that the homeomorphism consists in rotating all the cylinders at the same speed with respect to the metric, while Graf's notion needs to allow them them rotating at different speeds in the case of no vertices of valency one (or needs vertices of valency $1$, (which is equivalent to allow different speeds). Note that Graf's proof can not be adapted to prove our result.

	
\end{remark}

\begin{proof}[Proof of \Cref{thm:000}]
We use \Cref{not:cut}, \Cref{not:quot} and \Cref{not:cut_aut}. 

By \Cref{lem:collar} we can assume that there exist a collar for $\partial \Si$ such that $\phi'=\phi|_{\Si\setminus U}$ is periodic. Let $n$ be the order of $\phi'$. 
Let $\Si^{\phi'}$ be the quotient surface. By abuse the notation we are writing  $\Si^{\phi'}$ instead of $(\Si \setminus U)^{\phi'}$ since they are naturally homeomorphic.

We will construct the invariant graph $\G$ as in the proof of \Cref{lem:invariant}, as the preimage by the quotient map $p:\Si\to \Si^{\phi'}$ of an appropriate spine 
$\G^{\phi'}$ for $\Si^{\phi'}$. The proof consists in giving a metric for the spine in $\Si^{\phi'}$ such that the  pullback metric in the corresponding invariant graph in $\Si$ solves the problem. 

Let us see what conditions on the lengths of the edges of $\G^{\phi'}$ have to be imposed such that the pullback-metric in $\G$ defines a \tat metric adapted to $\phi$, 
that is, so that we have the equality $\gamma_p(\pi)=\phi(p)$ for the safe walks along $\G$  and such that the given signed \tat twist $\phi_{\G, \iota}$ has the same fractional Dehn twist coefficients as $\phi$. 

We define $$R_i:=|rot_{\partial^1} (\phi,C_i)|.$$ 
The \tat structure of $\G$ has to satisfy the equality $\gamma_p(\pi)=\phi(p)$ and moreover, the rotation number of $\phi_{\G, \iota}$ at $C_i$ has to be $R_i$. 
So, by \cref{eq:R_i}, we want that for every $i$ with $R_i\neq 0$ we have 
\begin{equation}
rot_{\partial^1} (\phi,C_i)\cdot length(\widetilde{\G}_i)\cdot \iota(i)=\pi,
\end{equation}

If $rot_{\partial}(\phi, C_i)$ equals $0$, by the definition of constant safe walk (see the end of \Cref{re:cut3}), we obtain no condition. 

By the definition of $R_i$ and the fact that both $rot_{\partial^1} (\phi,C_i)$ and $\iota(i)$ have the same sign, this equation becomes: 
\begin{equation}\label{eq:1}
R_i\cdot length(\widetilde{\G}_i)=\pi,
\end{equation}

Moreover, we want the metric on $\G$ to be invariant by $\phi'$ so it has to be the pullback of a metric on $\G^{\phi'}$. We denote by $\Si^{{\phi'}}_{\G^{{\phi'}}}$ the surface obtained by cutting $\Si^{\phi'}$ along 
$\G^{{\phi'}}$ and consider the gluing map $g_{\G^{{\phi'}}}:\Si^{{\phi'}}_{\G^{{\phi'}}} \to  \Si^{\phi'}$ analogously to \Cref{not:cut}.
We consider the lifting of $p:\Si\to \Si^{\phi'}$ to the cut surfaces and we denote it by $\widetilde{p}:\Si_\G\to \Si^{\phi'}_{\G^{\phi'}}$. 
We denote by $\widetilde{p(\G_i)}$ the preimage of $p(\G_i)$ by $g_{\G^{\phi'}}$. Since
$\widetilde{p}|_{\widetilde{\G}_i}:\widetilde{\G}_i\to \widetilde{p(\G_i)}$ 
is a $n:1$ covering map, we have the equality   
\begin{equation}
\label{eq:2}
length (\widetilde{\G}_i)=n\cdot length (\widetilde{p(\G_i)}).
\end{equation} 
Note that one can easily read $length(\widetilde{p(\G_i)})$ looking at the lengths of the edges of $p(\G_i)\cont\G^{\phi'}$. 
 
Putting (\ref{eq:1})-(\ref{eq:2}) together, we have that what we need is that the equality 
\begin{equation}
\label{eq:3} 
length(\widetilde{p(\G_i)})=\frac{\pi}{n\cdot R_i}
\end{equation}
holds for all $i$ with $R_i \neq 0$.

Next, we see that finding a metric spine $\G^{\phi'}\hookrightarrow \Si^{\phi'}$ containing all branching points, and whose lengths satisfy (\ref{eq:3}) for every $i$ with $R_i\neq 0$ we prove the theorem by taking $\G = p^{-1}(\G^{\phi'})$ with the pullback metric.

Since $\G^{\phi'}$ contains the branching points, the retraction of $\Si^{{\phi'}}$ to $\G^{{\phi'}}$ lifts to a retraction of $\Si$ to
the preimage $\G:=p^{-1}(\G^{{\phi'}})$. Hence $\G$ is a spine of $\Si$.

Then it is clear that the graph $\G:=p^{-1}(\G^{\hat{\phi}})$ is a signed \tat graph for $\iota$ defined as $\iota(i)=sign(rot_{\partial^1}(\phi,C_i))$. Let $\phi_{(\Sigma,\G,\iota)}$ be the induced signed \tat twist. 
This \tat structure on $\G$ induces by construction a rotation $\widetilde{\phi}_{(\Sigma,\G,\iota)}|_{\widetilde{\G}_i}$ of rotation number $|rot_{\partial^1}(\phi, C_i)-\left\lfloor  rot_{\partial^1}(\phi, C_i)\right\rfloor|$ in each $\widetilde{\G}_i$. It is conjugate to 
$\widetilde{\phi}|_{\widetilde{\G}_i}$ since they have the same rotation number (recall \Cref{cor:rot_red}). The orbits of $\widetilde{\phi}|_{\widetilde{\G}_i}$ are the fibres of 
$\widetilde{p}|_{\widetilde{\G}_i}$. By the choice of lengths, the orbits of the \tat rotation in $\widetilde{\G}_i$ are also the fibres of
$\widetilde{p}|_{\widetilde{\G}_i}$. Conjugation between homeomorphisms of $\MS^1$ preserves the cyclic order in $\MS^1$ of a point $p$ and its iterations. Then, since
$\widetilde{\phi}|_{\widetilde{\G}_i}$ and $\widetilde{\phi}_{(\Sigma,\G,\iota)}|_{\widetilde{\G}_i}$ are conjugate with the same orbits, they coincide. Then $\phi|_{\G}$ and $\phi_{(\Sigma,\G,\iota)}|_{\G}$ coincide.

By \Cref{re:iso_graph} we have that $\phi$ and $\phi_{(\Sigma,\G,\iota)}$ are isotopic since they coincide on $\G$.


We also see that, by the imposed metric,  
$rot_{\partial^1}(\phi_{(\Sigma,\G,\iota)},C_i)=rot_{\partial^1} (\phi,C_i)$ (just observe that the signed safe walk corresponding to the cylinder $\TG_i \times [0,1]$ {\em winds up} $R_i$ times around $\TG_i$). 
So, since $\phi$ and $\phi_{(\Sigma,\G,\iota)}$ coincide on a spine with a periodic homeomorphism and have the same fractional Dehn twist coefficients 
 we can conclude by 
\Cref{cor:iso_rot} that they are isotopic relative to the boundary. 

The restrictions of $\widetilde{\phi}$ and $\widetilde{\phi}_{(\Sigma,\G,\iota)}$ to each invariant cylinder $\widetilde{\Si}_i$ are conjugate by a homeomorphism that leaves its boundary  fixed pointwise. Then, these conjugation homeomorphisms glue together to a conjugation homeomorphism for $\phi$ and $\phi_{(\Sigma,\G,\iota)}$ that leaves $\G$ and $\partial\Si$ fixed pointwise. Then, $\phi$ and $\phi_{(\Sigma,\G,\iota)}$ are also conjugate as required in the statement.

Now we finish the proof finding such an invariant spine $\G$. 

First we prove the case $g:=genus(\Si^{\phi'})\geq 1$. To choose a spine in $\Si^{\phi'}$, we use a planar representation of $\Si^{\phi'}$ as a convex $4g$-gon in $\R^2$
with $r$ disjoint open disks removed from its convex hull. The sides of the $4g$-gon are labelled clockwise like 
$a_1b_1a_1^{-1}b_1^{-1}a_2b_2a_2^{-1}b_2^{-1} \ldots a_{g}b_{g}a_{g}^{-1}b_{g}^{-1}$, where edges labelled with the same letter (but different exponent)  are identified by an
orientation reversing homeomorphism. The number $r$ is the number of boundary components. We number the boundary components $C_i\cont\partial\Si$, $1\leq i\leq r$.
We denote by $d$ the arc $a_2b_2a_2^{-1}b_2^{-1} \ldots a_{g}b_{g}a_{g}^{-1}b_{g}^{-1}$.  
We consider $l_1$,...,$l_{r-1}$ arcs as in \Cref{fig:gn}. We denote by $c_1$,...,$c_{r}$ the edges in which $a_1^{-1}$ (and $a_1$) is subdivided, 
numbered according to the component $p(C_i)$ they enclose.
We consider the spine $\G^{{\phi'}}$ of $\Si^{{\phi'}}$ given by the union of $d$, $a_1b_1a_1^{-1}b_1^{-1}$ and the $l_i$'s. 
We construct $\G^{{\phi'}}$ so that it passes by all the branching points of $p$. 

We denote by $D$, $B_1$ and $C_i$ the lengths of $d$, $b_1$ and $c_i$ respectively. We will assume  all the $l_i$ and $b_1$ of the same length $L$.
Then the system (\ref{eq:3}) for this case  can be expressed as follows: 
$$2L+2C_1=\frac{\pi}{n\cdot R_1}$$
\begin{equation}
\label{eq:gn}
2L+2C_{i}=\frac{\pi}{n\cdot R_{i}}\ \ for \ i=2,...,r-1
\end{equation}
$$2L+2C_r+D=\frac{\pi}{n\cdot R_r},$$
which has obviously positive solutions $C_i$, $D$ after choosing for example $L=\min\{\frac{\pi}{4\cdot n\cdot R_{i}}\}$.
We assign $length(a_i)=length (b_i)=D/4(g-1)$ for $i>1$ to get a metric on $\G^{\phi}$. We consider the pullback-metric in $\G$. This finishes the case $genus(\Si^{\phi'})\geq 1.$

\begin{figure}[!ht] 
\centering
\includegraphics[scale=0.6]{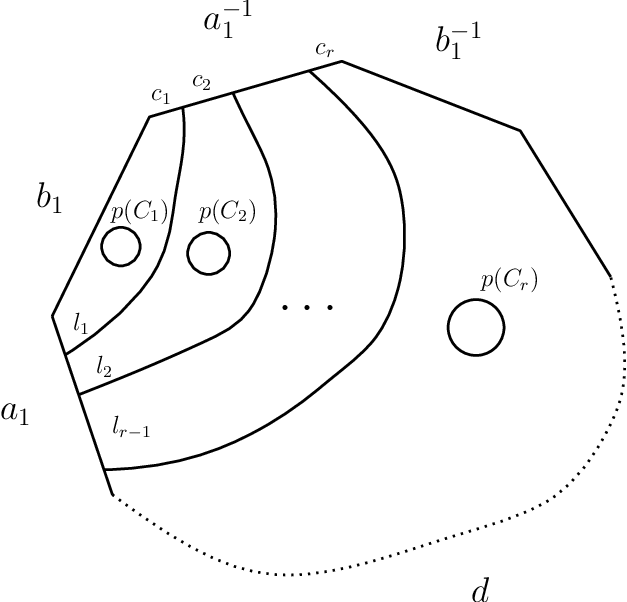}
\caption{Planar representation of the surface$\Si^\phi$ of genus$\geq 1$ and $r$ boundary components $p(C_1)$, ...,$p(C_r)$. Drawing of $l_1,...,l_r$ and $c_1,..,c_r$.}
\label{fig:gn}
\end{figure}
\begin{figure}[!ht] 
\centering
\includegraphics[scale=0.6]{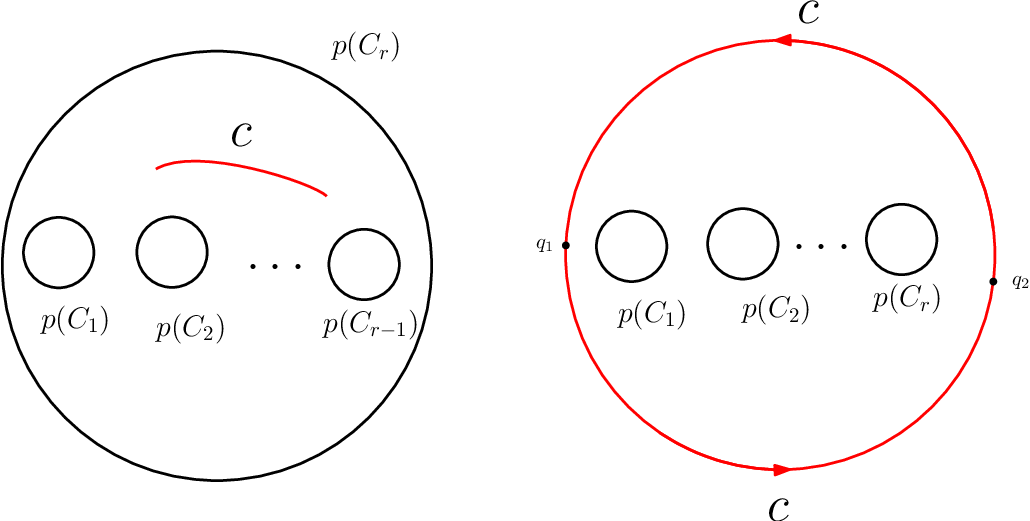}

\includegraphics[scale=0.6]{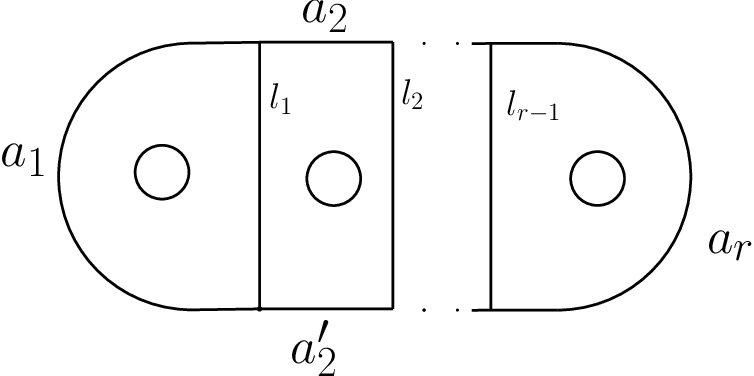}
\caption{In the first two pictures we have the disk with $r-1$ smaller disks removed. In the third one we forget the identification along the exterior and we draw a spine.}
\label{fig:g0}
\end{figure}

For the case $genus(\Si^{\phi'})=0$ we proceed a bit differently to choose the spine $\G^{\phi'}$. 
The surface $\Si^{\phi'}$ is a disk with  $r-1$ smaller disjoint disks removed. We cut the surface along an embedded segment that we call $c$ as we can see in the first 
image of \Cref{fig:g0}. Cutting along $c$ we get another planar representation of $\Si^{\phi'}$ as in the second image. The exterior boundary corresponds to $cc^{-1}$; 
we call the exterior boundary $P$ and denote by $q_1$ and $q_2$ the points in $P$ that come from the two extremes of $c$.
We look at the graph of the third picture in \Cref{fig:g0}. We have drawn $r-1$ vertical 
segments $l_1$,...,$l_{r-1}$ so that $P$ union with them contains all branch points and is a regular retract of the disk enclosed by $P$ minus the $r$ disks. 

The rest of the proof follows by cases on the number of boundary components $r$ and the number of branch points.

If $r=1$ and there are no branch points or $1$ branch point, then $\Si$ is a disk (this follows from the Hurwitz formula) which is not covered by the  statement of the theorem. If there are at least $2$ branch points, we can get that two branch points lie in $q_1$ and $q_2$ so that $\G$ has no univalent vertices. In this case we set $length(c)=\frac{\pi}{2 n R_1}$.

Suppose now $r=2$.  If there are no branch points, then $\Si$ is a cylinder which is not included in the statement of this theorem.

If there is at least $1$ branch point we consider two cases, namely $R_1 = R_2$ and $R_1 < R_2$.

In the  case $R_1=R_2$, we choose the graph depicted on the right hand side of  \Cref{fig:genus0r2}. That is, $q_1$ and $q_2$ are exactly $a_1 \cap a_2$. In this case we do not care about the location of the branch point as long as it is contained in the graph. In this case we set $lenght(c)=length(l_1)= \frac{\pi}{2nR_1}$

In the  case $R_1<R_2$, we choose the graph depicted on the left hand side of  \Cref{fig:genus0r2} and we choose the branch point to lie on $q_1$, this way, since $q_1$ is the only vertex of valency $1$, we get that the preimage of this graph by $p$ does not have univalent vertices. We set the lengths, $length(l_1)=\frac{\pi}{n R_1}$ and $length(c)=(\frac{\pi}{nR_2} - \frac{\pi}{n R_1})/2$.

\begin{figure}[!ht]
\centering
\includegraphics[scale=0.6]{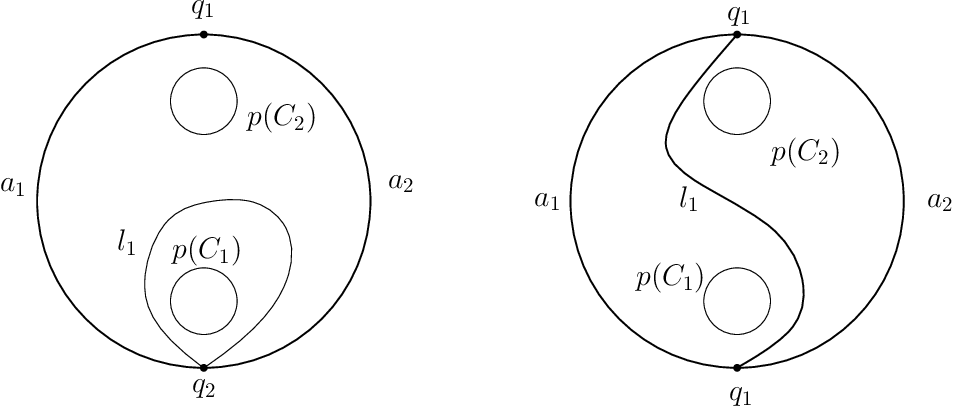}
\caption{On the left, the case $R_1<R_2$. On the right the case when $R_1=R_2$.}
\label{fig:genus0r2}
\end{figure}

Suppose now $r>2$.

We are going to assign lengths to every edge in \Cref{fig:g0} and decide how to divide and glue $P$ in order to recover $\Si^{\phi'}$. This means that we are 
going to decide the position of $q_1$ and $q_2$ in $P$, relative to the position of the ends of the $l_i$'s, in order to get a suitable metric spine of the quotient surface $\Si^{\phi'}$.  

To every vertical interior segment $l_j$ we assign the same length 
$$L<min\{\frac{\pi}{2(n\cdot R_i)}\}.$$
We look at the segments 
$a_1$, $a_2$, $a_2'$,...,$a_{r-1}$, $a_{r-1}'$, $a_r$ in which $P$ is divided by the vertical segments (see \Cref{fig:g0}) and give lengths 
$A_1$, $A_2$, $A_2'$,...,$A_{r-1}$, $A_{r-1}'$, $A_r$. The following system corresponds to (\ref{eq:3}) for this case:
$$A_1+L=\frac{\pi}{n\cdot R_1}$$
\begin{equation}
\label{eq:g0}
A_i+A_{i}'+2L=\frac{\pi}{n\cdot R_{i}}\ \ for\ \  i=2,...,r-1
\end{equation}
$$A_r+L=\frac{\pi}{n\cdot R_r}.$$
It has obviously positive solutions $A_i$. We choose $A_i=A_i'$ for $i=2,...,r-1$. 

In order this distances can be pullback to the original graph $\Si^{\phi'}$, there is one equation left: we have to impose equal length to the two paths $c$ and $c^{-1}$, or in other words to place $q_1$ and $q_2$ dividing $P$ in two segments of equal length. 

If $\phi$ has at least two branching points, we choose $q_1$ and $q_2$ to be any two of the branching points. Then, we can choose the metric and the vertical segments so that 
$q_1$ and $q_2$ are the middle points of $a_1$ and $a_r$.
If we identify the two paths $c$ and $c^{-1}$ joining $q_1$ and $q_2$ then we recover $\Si^{\phi'}$, and we get a metric graph on it.
Then, the preimage by $p$ of the resulting metric graph gives a metric graph. We claim that this graph has no univalent vertices. Indeed, a univalent vertex of this graph has to be the preimage of univalent vertices of the graph below, which are only $q_1$ and $q_2$, which are branching points. By 
\Cref{galois} all their preimages are ramification points and then they are not univalent vertices. 
The metric induces a \tat structure in the graph by construction. Now, we finish the proof as in the case $genus(\Si^{\phi'})\geq 1$.

If $\phi$ has no ramification points or only one, in the previous construction we could obtain univalent vertices at the preimages of $q_1$ and $q_2$.
So we need to do some changes in the assignations of lengths and in the positions of the vertical segments relative to $q_1$ and $q_2$ in order that the extremes 
$q_1$ and $q_2$ of $c$ coincide with vertices of the graph in $\Si^{\phi'}$. 

Assume that $A_1\geq A_2 \geq \cdots \geq A_r$. 

We choose $q_1:=a_1\cap a_2$. Let $q_2$ be the antipodal point (so $q_1$ and $q_2$ divides $P$ in two paths of equal length). 
If $q_2$ is a vertex we have finished. If it is not, then it is on a segment $a'_i$, for some $i=2,...,r_1$. 
We redefine $A'_i:=d(a'_i\cap a'_{i-1},q_2)$ and $A_i:=A_i+d(q_2,a'_i\cap a'_{i+1})$.  Now the antipodal point of $q_1$ is $a_i\cap a_{i+1}$. 
We redefine $q_2:=a_i\cap a_{i+1}$ and identify orientation reversing the two paths joining $q_1$ and $q_2$ to recover $\Si^\phi$.
 
Now the pullback of the resulting metric graph has no univalent vertices and gives a \tat structure since the corresponding system of equations is satisfied.
\end{proof}

\begin{remark}
Observe that in the case that $genus(\Si^{\hat{\phi}})\neq 0$ or the quotient map by $\hat{\phi}$ has at least two branching points, 
we have found a spine $\G$ of $\Sigma$ such that for any homeomorphism $\psi$ which fixes pointwise the boundary and is boundary-free isotopic to $\hat{\phi}$, 
there is a signed \tat structure on $\G$ (that is, a metric and a sign function $\iota$) such that $\psi$ is isotopic relative to the boundary to the corresponding 
signed \tat homeomorphism. In other words, there is a {\em universal} spine which may be endowed of signed \tat structures representing all boundary fixed isotopy classes 
of homeomorphisms which are boundary-free isotopic to $\phi$. 
\end{remark}

Now we state the relative case in a simple way, skipping the obvious strengthenings similar to the previous theorem:

\begin{theorem}\label{thm:rel_signed}
Let $\Si$ be an oriented surface with non-empty boundary. Let $\partial^1\Si$ be a non empty union of boundary components. Let $A$ be the union of the boundary 
components not contained in $\partial^1\Si$. Let ${\phi}$ be an orientation preserving homeomorphism fixing pointwise $\partial^1\Si$ and boundary-free isotopic to a 
periodic one $\hat{\phi}$. Then, there exists a signed \tat spine $(\G,A,\iota)\hookrightarrow\Sigma$ such that $\phi_{(\Sigma,\G,A,\iota)}$ is isotopic relative to $\partial^1\Si$ to $\phi$. Moreover, if $\phi$ is periodic outside a collar of $\partial^1\Si$, we have also that  $[\phi]_{\partial,\phi|_{\partial}}=[\phi_{(\Sigma,\G,A,\iota)}]_{\partial,\phi|_{\partial}}$.
\end{theorem}

\begin{proof}
Apply Alexander's trick (see \Cref{def:alex}) to the boundary components in $A$ in order to obtain a larger surface and a homeomorphism fixing pointwise the boundary. Construct a signed \tat graph
inducing this homeomorphism like in the proof of \Cref{thm:000}. We can always get that this signed \tat graph contains as vertices the centers of the disks added by Alexanders's trick. Now apply an $\epsilon$-blow up (recall \Cref{def:blow}) to these vertices to get the desired signed relative \tat graph. 
\end{proof}
\begin{remark}\label{re:v3} Note that this produces a relative \tat graph where all the vertices in the boundary have valency at most 3.
\end{remark}

\subsection{Periodic \tat twists and periodic representatives in $MCG(\Si)$.}\label{sub:comp_per}
In this section we introduce a periodic twist induced by a \tat graph. We prove  Theorem 5.25 as a corollary of the previous section, which shows  that periodic \tat twists give rise to all truly periodic homeomorphisms (leaving at least one boundary component invariant) up to isotopy and 
conjugacy.




\begin{definition}[Periodic \tat twist]\label{def:phi_Gper}
Replacing the map (\ref{eq:psi}) in \Cref{def:phi_Gtau} by 
$$\psi_{i}:\widetilde{\G}_i\times I\longrightarrow\widetilde{\G}_i\times I$$
$$(p,s)\mapsto (\widetilde{\gamma}^{+}_p(\pi),s)$$
we get a truly periodic homeomorphism that we denote by $\phi^{per}_{\Sigma,\G,A,\iota}$.
We call it the induced \textit{periodic \tat twist}.
\end{definition}

\begin{remark}\label{rem:conju}For two different product structures of the cylinders, the induced periodic \tat  twists are conjugate
by a homeomorphism that fixes $\G$.

For two different embeddings in $\Si$, the periodic \tat twists are conjugate by the same homeomorphism of $\Sigma$ that relates the two embeddings. 
\end{remark}

Note that obviously $\phi^{per}_{\Sigma,\G,A,\iota}$ is boundary-free isotopic to $\phi_{(\Sigma,\G,A,\iota)}$ and coincides along $\G$.  

Given a periodic homeomorphism $\phi$, we can clearly choose a representative $\phi'$ of $[\phi]$ that leaves $\partial^1 \Si$ fixed pointwise (by isotoping $\phi$ near  $\partial^1 \Si$ until it is the identity on it). Then, we can find a signed \tat graph $(\G, A,\iota) $ embedded in $\Si$ to represent it using \Cref{thm:000} and \Cref{thm:rel_signed}.  Then, we can consider the periodic homeomorphism $\phi^{per}_{\Sigma,\G,A,\iota}$. Note that we can always get that $\phi'$ has all its fractional Dehn twist coefficients  positive, which means $\iota\equiv +$. Then, we get 
the following theorem by using Remark \ref{re:cyl}:

\begin{theorem}\label{thm:rel}
Let $\Si$ be a connected surface with non-empty boundary which is not a disk or a cylinder. Let $\phi$ be an orientation preserving periodic homeomorphism of $\Si$ that leaves 
(at least) one boundary component invariant. Let $A$ be the set containing all boundary components that are not invariant
by $\phi$. Then there exists a relative \tat graph 
$(\G, A)$ embedded in $(\Si,A)$, which is invariant by $\phi$, such that
\begin{enumerate}[(i)]
\item We have the equality of boundary-free isotopy classes 
$[\phi_{(\Sigma,\G,A,\iota)}^{per}]_{A, \phi|_A}=[\phi]_{A, \phi|_A}$ (see  \Cref{not:iso}). 
\item  the homeomorphism $\phi$  is conjugate to $\phi_{(\Sigma,\G,A,\iota)}^{per}$ by a homeomorphism that fixes $\G$.
\end{enumerate}
\end{theorem}

\begin{remark}
Observe that in order to represent all truly periodic homeomophisms we do not need the extension to signed \tat graphs in \Cref{def:tat_tau}; The first author 
original definition, given in \Cref{def:tat} is enough.
\end{remark}

The following corollary recovers a known result giving an elementary proof of it. See, for example, \cite[Theorem 7.14]{Farb} and the introduction to section 7.4 in that same reference for a different proof.

\begin{corollary}\label{cor:finite_conju}
	Let $\Si$ be the surface of genus $g$ and $r$ boundary components. There are finitely many conjugacy classes of finite-order mapping classes in $MCG(\Si)$.
\end{corollary}

\begin{proof}
	Observe that there are only a finite amount of spines without vertices of valency $1$. By \Cref{thm:rel}, these are enough to model all periodic mapping classes. By \Cref{rem:conju}, two distinct embeddings of the same graph produce conjugate periodic homeomorphisms and the result follows. 
\end{proof}
\section{Pseudo-periodic homeomorphisms.}

\label{sec:pseudo}
\subsection*{Definition and conventions}
We recall some definitions and fix some conventions on pseudo-periodic homeomorphisms of surfaces with boundary.

\begin{definition}\label{def:pseudo}
A homeomorphism $\phi:\Si \rightarrow \Si$ is pseudo-periodic if it is isotopic to a homeomorphism satisfying that there exists a finite collection of disjoint simple closed curves $\mathcal{C}$ such that 
\begin{enumerate}
\item $\phi(\mathcal{C})= \mathcal{C}$
\item $\phi|_{\Si \setminus \mathcal{C}}$ is boundary-free isotopic to a periodic homeomorphism. 
\end{enumerate}
This system of cut curves $\mathcal{C}$ is called  a system of cut curves subordinated to $\phi$.
\end{definition}

The following theorem is due to Nielsen (see \cite[Section 15]{Niel1}): 

\begin{theorem}\label{theo:min} Given a pseudo-periodic homeomorphism $\phi$, there exists a minimal system of cut curves $\calC$. In particular, none of the connected components of $\Si \setminus \mathcal{C}$ is neither a disk nor an annulus.  A minimal system of cut curves  is unique up to isotopy.  
\end{theorem}



\begin{remark}[Quasi-Canonical Form]\label{re:stan} Given a system of cut curves $\calC=\{\calC_i\}$ subordinated to $\phi$, it is clear that there exists an isotopic homeomorphism $\phi'$ that admits annular neighbourhoods $\calA=\{\calA_i\}$ of the curves in $\mathcal{C}$ with $\calC_i \subset \calA_i$ such that
\begin{enumerate}
\item $\phi'(\calA)=\calA$, 
\item the map $\phi'|_{\overline{\Si\setminus\calA}}$ is periodic.
\end{enumerate}

Moreover, in the case $\phi$ fixes pointwise some components $\partial^1\Si$ of the boundary $\partial \Si$, we can always find an isotopic homeomorphism $\phi''$ relative to $\partial^1 \Si$ that coincides with a homeomorphism satisfying (1) and (2) outside a collar neighborhood $U$ of $\partial^1 \Si$. We may assume that there exists an isotopy connecting $\phi$ and $\phi''$ relative to $\partial^1\Si$.

We say $\phi'$ and $\phi''$ are \emph{quasi-canonical forms} for $\phi$ with respect to the set $\calC$ of cut curves. 
\end{remark}

\begin{definition}[Canonical Form]\label{def:can} 
A quasi-canonical form for a quasi-periodic homeomorphism $\phi$ with respect to a minimal system of cut curves is called a \emph{canonical form} for $\phi$. 
  
\end{definition}

\begin{remark}\label{re:un_C} Note that the uniqueness up to isotopy of a minimal system of cut curves $\mathcal{C}$ (see  \Cref{theo:min}) implies that $\phi|_{\mathcal{A}}$ is unique up to conjugacy where $\mathcal{A}$ is the collar neighbourhood of $\mathcal{C}$ in the definition of canonical form.  
\end{remark}

\begin{notation}\label{not:D_an}
Let $m,c \in \R$. We denote by $\calD_{m,c}$ the homeomorphism of $\MS^1 \times I$ induced by $(x,t) \mapsto (x +mt +c,t)$ (we are taking $\MS^1 = \R/ \Z$). Observe that 
\begin{equation}\label{eq:Dehn_sum}\calD_{m,c}\circ \calD_{m',c'} = \calD_{m+m',c+c'},\end{equation}
\begin{equation}\calD_{m,c}^{-1}=\calD_{-m,-c}.\end{equation}
In any case, in this work we will always have $m \in \Q$.
\end{notation}

\begin{lemma}[Linearization. Lemma 2.1 in \cite{MM}]\label{lem:lin} 
Let $\calA_1$ be an annulus and let $\phi: \calA_1 \to \calA_1$ be a homeomorphism that does not exchange boundary components. Suppose that $\phi|_{\partial \calA_1}$ is periodic. Then, after an isotopy of $\phi$ preserving the action at the boundary, there exists a parametrization $\eta:\MS^1 \times I\to \calA_1$ such that $$\phi=\eta\circ \calD_{-m,-c}\circ\eta^{-1}$$ for some $m,c\in \Q$. 
\end{lemma}

\begin{remark}\label{re:lin} In the case $\phi|_{\partial \calA_1}$ is the identity, we have that $$\phi= \eta\circ \calD_{m,0}\circ\eta^{-1}$$
for some $m\in \Z$. 

In the case $m=1$, that is the conjugacy class of $\calD_{1,0}$, is what is called in the literature negative Dehn twist  (compare to convention in \Cref{rem:convention}). 

The same name of negative Dehn twist is used for the  extension of $\phi$  by the identity to a homeomorphism of a bigger surface . 
\end{remark}

\begin{remark}\label{re:unic_sc}
 Note that $m$ and $c$ are completely determined in Lemma \ref{lem:lin}. The parameter $c$ (which only matters modulo $\Z$) equals the rotation number of $\phi|_{\eta(\MS^1\times\{0\})}$. The sum $c+m$ modulo $\Z$ equals the rotation number of $\phi|_{\eta(\MS^1\times\{1\})}$ measured with the orientation induced by the direct identification with $\MS^1 \times \{0\}$ (and not as boundary of the annulus). Moreover, the parameter $m$ is completely determined since $\mathcal{D}_{m,c}$ and $\mathcal{D}_{m',c}$ with $m\neq m'$ and $m\equiv m'\ mod\ \Z$ are never isotopic relative to the boundary. 
\end{remark}

\begin{definition}[Screw number]\label{def:screw_number}
Let  $\{\calA_i\}_{i=1}^\alpha$ be a set of annuli cyclically permuted by a homeomorphism $\phi$, i.e. $\phi(\calA_i) = \calA_{i+1}$ and $\phi(\calA_\alpha)=\calA_1$. Denote $\calA:=\bigcup_{i=1}^\alpha \calA_i$. Assume that $\phi|_{\partial\calA}$ is periodic and let $n$ be its order. 
 
By \Cref{re:lin}, $\phi^n|_{\calA_i}$ equals a conjugate to $\calD_{e_i,0}$ for a certain $e \in \Z$. We define 
\begin{equation}
\label{eq:screw}s(\calA_i):=-\frac{e}{n} \theta
\end{equation} 
where $\theta=\alpha$ if $\phi^{\alpha}|_{\calA_i}$ does not interchange boundary components and $\theta=2\alpha$ in the 
other case. 
We call $s(\calA_i)$ the {\em screw number} of $\phi$ at $\calA_i$.

Let $\calC=\{\calC_i\}$ be a system of cut curves subordinated to a pseudo-periodic homeomorphism $\phi$. Assume $\phi$ is in quasi-canonical form (see \Cref{re:stan}) with respect to $\calC$. We define the screw numbers $s(\calC_i)$ as the screw number of the corresponding annulus. 
\end{definition}

\begin{remark}\label{re:screw_number}
Let  $\{\calC_i\}_{i=1}^\alpha$ be curves cyclically permuted by a homeomorphism $\phi$ and that are a subset of a system of cut curves subordinated to a pseudo-periodic homeomorphism $\phi$ that we assume in quasi-canonical form (see \Cref{re:stan}).  
It can be checked that these numbers are well defined independently of the choice of $\calC$ and $\calA$. 
\end{remark}

\begin{remark} 
Compare \Cref{def:screw_number} with \cite[p.4]{MM} p.4 and with \cite[Definition 2.4]{MM}. 
The original definition is due to Nielsen \cite[Section 12]{Niel1}. 
\end{remark}

\begin{remark}\label{rem:convention_adj}
By Corollary 2.2 in \cite{MM} we have that $\calD_{s,c}$ has a screw number equal to $-s$. In particular the negative Dehn twist $\calD_{1,0}$ has screw number $-1$.
\end{remark}

We start with an easy lemma that is important for \Cref{thm:mixed_model} 
\begin{lemma}\label{lem:qp} 
Let  $\{\calA_i\}_{i=1}^\alpha$ be a set of annuli and $\phi$ as in \Cref{def:screw_number} permuting them cyclically. Suppose that 
$\phi|_{\calA_1}^\alpha$ does not interchange boundary components. Then, after an isotopy of $\phi$ preserving the action
at all the boundary components, there exist coordinates $$\eta_{i}:\MS^1 \times I\to \calA_i$$ for the annuli in the 
orbit such that 
$$ \eta_{j+1}^{-1}\circ\phi\circ\eta_{j}=\calD_{-m/\alpha,-c/\alpha} $$ where $m$ and $c$ are associated to
$\calA_1$ and $\phi^\alpha|_{\calA_1}$ as in \Cref{lem:lin}. 
\end{lemma}

\begin{proof}

Isotope $\phi$ if necessary in order to take a parametrization of $\calA_1$, associated to $\phi^{\alpha}:\calA_1 \rightarrow \calA_1$ as in \Cref{lem:lin}. Denote this parametrization by  $\eta_1:\MS^1 \times I\to \calA_1$.

Define recursively $\eta_j:=\phi\circ\eta_{j-1}\circ \calD_{m/\alpha,c/\alpha}$ (see \Cref{not:D_an}).
Then, we have
$$\eta_{j+1}^{-1}\circ\phi\circ\eta_{j}=\calD_{-m/\alpha,-c/\alpha}.$$
Since for every $j$ we have that $\eta_j=\phi^{j-1}\circ\eta_1\circ \calD_{m(j-1)/\alpha,c(j-1)/\alpha}$ we have also that 
$$\eta_1^{-1}\circ\phi\circ \eta_{\alpha}=\eta_1^{-1}\circ\phi\circ\phi^{\alpha-1}\circ\eta_1\circ \calD_{m(\alpha-1)/\alpha,c(\alpha-1)/\alpha}=\calD_{-m/\alpha,-c/\alpha}.$$
\end{proof}

\begin{remark}\label{re:sk} 
In particular, in the setting of this lemma, we have the equality $m=s(\calA_1)$ since by (\ref{eq:Dehn_sum}) we have 
$e=m\cdot n/\alpha$ and in this case we have $\theta=\alpha$ in  (\ref{eq:screw}).

Moreover, after this proof we can check that $\eta_k^{-1} \circ\phi^\alpha\circ\eta_k=\calD_{-s,-c}$ to see that the screw number $s=s(\calA_i)$ and the parameter $c$ modulo $\Z$ of \Cref{lem:qp} only depend on the orbit of $\calA_i$. 

\end{remark}

\subsection{Gluings and boundary Dehn twists}
In this section we introduce some notions and examples that will be used in the rest of the paper. 
We begin with two easy remarks.

\begin{remark}\label{re:g2}
Given a homeomorphism $\phi$ of a surface $\Si$ with $\partial \Si\neq \emptyset$. Let $C$ be a connected component of $\partial \Si$. Let $A$ be a cylinder parametrized by $\eta:\MS^1 \times I\to A$. We glue $A$ with $\Si$ by an identification $f:\MS^1 \times \{0\}\to C$. Then, we can extend $\phi$ \emph{trivially} along $A$ by defining a homeomorphism $\bar{\phi}$ of $\Si\cup_f A$ as $\bar{\phi}|_{\Si}=\phi$ and for any $p\in A$ we set $\phi(p)=\phi(f\circ p_1\circ\eta^{-1}(p))$ where $p_1$ is the canonical projection from $\MS^1 \times I$ to $\MS^1=\MS^1 \times \{0\}$.
\end{remark}

\begin{remark}\label{re:g1}
Given a homeomorphism $\phi$ of a surface $\Si$ with $\partial \Si\neq \emptyset$. Let $C$ be a connected component of $\partial \Si$. Let $A$ be a compact collar neighborhood of $C$ (isomorphic to $I\times C$) in $\Si$.  Let $\eta:\MS^1 \times I\to A$ be a parametrization of $A$, with $\phi(\MS^1 \times \{1\})=C$.

Then, there exists a homeomorphism $\phi'$ isotopic to $\phi$ relative to the boundary such that 
\begin{itemize}
\item the restriction to $A$ satisfies $p_2\circ\eta^{-1}\circ\phi'|_{A}\circ\eta(t,x)=p_2\circ\eta^{-1}\circ\phi'|_{A}\circ\eta(t',x)$ for all $t, t'\in I$, where $p_2:\MS^1 \times I\to \MS^1$ is the canonical projection. 
\end{itemize}
\end{remark}

\begin{definition}[Boundary Dehn twist]\label{def:boundary_dehn}
Let $C$ be a component of $\partial \Si$ and let $A$ be a compact collar neighbourhood of $C$ in $\Si$. Suppose that $C$ has a metric and total length  is equal to $\ell$. Let $\eta:\MS^1 \times I\to A$ be a parametrization of $A$, such that  $\eta|_{\MS^1 \times \{1\}}:\MS^1 \times \{1\} \to C$ is an isometry. Suppose that $\MS^1$ has the metric induced from taking $\MS^1 = \R/ \ell\Z$ with $\ell \in \R_{>0}$ and the standard metric on $\R$. A \emph{boundary Dehn twist} of length $r \in \R_{>0}$ along $C$ is a homeomorphism $\calD^\eta_{r} (C)$ of $\Si$ such that:
\begin{enumerate}
\item it is the identity outside $A$
\item the restriction of $\calD^\eta_{r} (C)$ to $A$ in the coordinates given by $\eta$ is given by $(x,t) \mapsto ( x+ r\cdot t, t).$
\end{enumerate}
The isotopy type of $\calD^\eta_{r} (C)$ by isotopies fixing the action on $\partial \Si$ does not depend on the parametrization $\eta$. When we write just $\calD_{r}(C)$, it means that we are considering a boundary Dehn twist with respect to {\em some} parametrization $\eta$.
\end{definition}

\begin{example}\label{ex:comp_tat}
We can restate the $\ell-$\tat property in terms of boundary Dehn twists of length $\ell$ as follows. Let $(\Si, \G)$ be a thickening surface of a metric ribbon graph $\G$. Let $g_\G: \Si_\G\to \Si$ be the gluing map. Consider the pull back metric on $g^{-1}(\G)$. 
Denote by $\calD_{\underline{\ell}}$ the composition of the boundary Dehn twists $\calD_{\ell}$ along each $\TG_j \subset \TG$. Then $(\Si, \G)$ holds the $\ell$-\tat property if and only if $\calD_{\underline{\ell}}$ is compatible with the gluing $g_\G$. We see from this that the lenghts of $\TG_j$ are in $\ell\Q_+$.
\end{example}

\section{Mixed \tat graphs and twists and pseudo-periodic homeomorphisms.}\label{sec:intro_mixed}
In this section we introduce the notion of mixed \tat graphs and the homeomorphisms that they induce which we call mixed \tat twists. Mixed \tat graphs and twists are generalizations of \tat graphs and twists, but they are able to model pseudo-periodic homeomorphisms. Our main result is a realization theorem for a certain class of homeomorphisms whose isotopy classes are representable by a \tat twists (see  \Cref{thm:mixed_model}). The class contains the monodromies of arbitrary plane branches, generalizing the construction for branches with 2 Puiseux pairs in~\cite{Camp1}.  

Before introducing the definition of mixed \tat graphs and twists, we analyze 
in \Cref{sec:restricted} the structure of the class of pseudo-periodic 
homeomorphisms refrerred above, and show how to codify them using {metric 
spines}. We think this section makes the reading easier, but strictly speaking 
the reader could skip it now, read the definition of mixed \tat graphs and 
twists in \Cref{sec:mixed} and \Cref{sec:mixed_homeo}, read the statement of 
\Cref{thm:mixed_model} in \Cref{sec:final}, and come back to 
\Cref{sec:restricted} for its proof. In a sequel article by the third author 
and B. Sigurdsson~\cite{Bal} it is proved a general realization theorem for 
pseudo-periodic mapping classes in terms of the mixed \tat graphs defined here.

\subsection{A restricted type of pseudo-periodic homeomorphisms}
\label{sec:restricted}

In this section we work with a restricted type of pseudo-periodic homeomorphisms and give a natural construction of an embedded metric filtered graph that is a retract of the surface and that codifies the homeomorphism up to isotopy relative to the boundary. It is important, however, to notice that, unlike in the case of the previous sections, the graph is not left invariant by the homeomorphism (in fact this would force the homeomorphism to be boundary-free periodic).  
This study will lead to the more general definition of mixed \tat graph and the induced \tat twist developed in the next sections.

We start giving the hypothesis of our homeomorphism. 


Let $\phi$ be a pseudo-periodic homeomorphism of a surface $\Si$ with $\partial\Si\neq \emptyset$. 
Let $\calC$ be a system of cut curves for $\phi$ as in \Cref{def:pseudo}. Let $G(\phi, \Si)$ be a graph constructed as follows:
\begin{enumerate}
\item It has a vertex for each connected component of ${\Si}\setminus\calC$.
\item There are as many edges joining two vertices as curves in $\calC$ intersect the two surfaces corresponding to those vertices.
\end{enumerate}

\textbf{Assumptions on $\phi$}:\label{page:assump}
\begin{enumerate}
\item the graph $G(\phi, \Si)$ is a tree and
\item the screw numbers are all non-positive. 
\item We assume that 
\begin{itemize}
\item[(3')] it leaves at least one boundary component pointwise fixed,             
\item[(3'')] the fractional Dehn twist coefficients along at least one of these fixed-boundary components is positive (we extend in the obvious way the notion of fractional Dehn twist coefficients in~\Cref{def:gen_rot} 
and~\Cref{lem:equal_rot} to pseudo-periodic homeomorphisms by considering the restriction of the homeomorphism to the connected component in $\Si \setminus \calA$ that contains $C$).
\end{itemize}
\end{enumerate}
We denote by $\partial^1\Si$ the union of some, at least one, connected components of $\partial\Si$ contained in a single connected component of $\Si \setminus \calC$ and that are fixed pointwise by $\phi$ and have positive fractional Dehn twist coefficient.

We will obtain a metric spine codifying $[\phi]_{\partial \Si,\phi|_{\partial \Si}}$ (recall \Cref{not:iso}).

\begin{remark} Observe that $G(\phi, \Sigma)$ being a tree implies that for any invariant orbit of annuli associated to 
the system of cut curves, we are under the hypothesis of \Cref{lem:qp}. 
\end{remark}

\begin{figure}
\includegraphics[width=60mm]{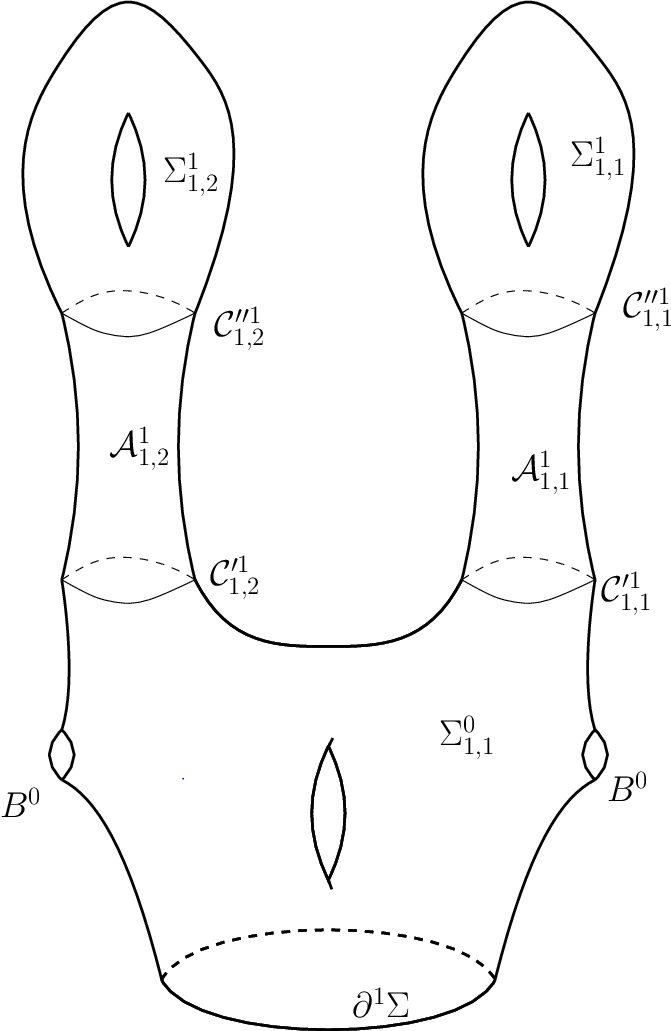}
\caption{This figure helps understand \Cref{not:qp2}. It is the surface of genus $3$ and $3$ boundary components. Suppose that the notation is induced by an homeomorphism that rotates by $\pi$ around the $z$-axis composed by some Dehn twist around the annuli $\calA$ and isotoped to be the identity on the bottom boundary. The boundaries $B^0$ are defined after the notation as those in $\Si^1$ but not in $\partial^1 \Si$.}
\label{fig:rest}
\end{figure}

\begin{notation}\label{not:qp2}
Figure \ref{fig:rest} may help.  
We assume $\phi$ is in the quasi-canonical form of \Cref{re:stan} with respect to $\partial^1\Si$. We denote by $\hat{\Si}$ the closure of $\Si\setminus\calA$ in $\Si$.
  
Let $v(G(\phi, \Si))$ be the set of vertices of $G(\phi, \Si)$. We choose as root of $G(\phi,\Si)$ the vertex $v \in v(G(\phi, \Si))$ corresponding to the connected component of $\Si\setminus \calA$ that contains $\partial^1\Si$. We say that $G(\phi, \Si)$ is rooted at $v$.

Since $\hat{\phi}$ permutes the surfaces in $\hat{\Si}$, it induces a permutation of the set $v(G(\phi, \Si))$ which we denote by $\sigma_\phi$.
 
We begin filtering the set $v(G(\phi, \Si))$:
\begin{enumerate}
\item Denote the vertex chosen as the root by $v^0_{1,1}$. Let $V^0:=\{v^0_{1,1}\}$.
\item Let $d: G(\phi,\Si) \rightarrow \Z_{\geq 0}$  be the distance function to $V^0$, that is, $d(v)$ is the number of edges of the smallest bamboo in $G(\phi,\Si)$ that joins $v$ with $V^0$. Let $V^i:= d^{-1}(i)$. Observe that the permutation $\sigma_\phi$ leaves the set $V^i$ invariant. There is a labelling of $V^i$ induced by the orbits of $\sigma_\phi$: suppose it has $\beta_i$ different orbits. For each $j=1, \ldots, \beta_i$, we label the vertices in that orbit by $v^i_{j,k}$ with $k=1, \ldots, \alpha_j$ so that $\sigma_\phi(v^i_{j,k})=v^i_{j,k+1}$ and $\sigma_\phi(v^i_{j,\alpha_j})=v^i_{j,1}$. 
\end{enumerate}

We accordingly give names to distinct parts of the surface $\Si$ (see Figure \ref{fig:ex_grafo} for an example):
\begin{itemize}
	\item Denote by $\Si^i_{j,k}$ the surface in $\hat{\Si}$ corresponding to the vertex $v^i_{j,k}$.
	\item Denote by $\Si^i$ the union of the surfaces corresponding to the vertices in $V^i$. Note that $\Si^0$ equals $\Si^{0}_{1,1}$.
	\item Denote by $\calA^{i+1}_{j,k}$ the only annulus in $\calA$ that intersects both $\Si^{i+1}_{j,k}$ and $\Si^{i}$. Observe that if there were more than one such annulus, $G(\phi, \Si)$ would not be a tree.
	\item Denote by $\calC'^{i+1}_{j,k}$ the boundary component of $\calA^{i+1}_{{j,k}}$ that lies in $\partial \Si^{i}$ and by $\calC''^{i+1}_{j,k}$ the other boundary component. 
	\item Denote $\calA^{i+1}$ the union of all the annuli that intersect $\Si^{i+1}$ and $\Si^{i}$ and define analogously $\calC'^{i+1}$ and $\calC''^{i+1}$.
	\item  We also define recursively $$\Si^{\leq 0}:= \Si^0, \ \ \Si^{\leq i+1}:= \Si^{\leq i}\cup \calA^{i+1} \cup \Si^{i+1}.$$
\end{itemize}

We recall that  $\alpha^{i}_j$ is the smallest positive number such that $\phi^{\alpha^{{i}}_j}(\calA^{i+1}_{j,k})=\calA^{i+1}_{j,k}$, and in consequence the least such that $\phi^{\alpha^{{i}}_j}(\calC'^{i+1}_{j,k})=\calC'^{i+1}_{j,k}$ and $\phi^{\alpha^{{i}}_j}(\calC''^{i+1}_{j,k})=\calC''^{i+1}_{j,k}$.
\end{notation} 

\textbf{Construction of the metric relative spines for $\Si^{\leq i}$.}

We start now recursively building an embedded metric spine  $(\G^{\leq i},B^{\leq i})$ for $\Si^{\leq i}$ codifying $\phi_{\Si^{\leq i}}$. The way in which the metric spine codifies the automorphism will become clear along the construction  
The starting point is the following: 
\begin{itemize}
\item By \Cref{thm:rel}, there exists a $\pi$-relative \tat graph $(\Lambda^0, B^0)$ embedded in $\Si^0$ 
\begin{itemize}
\item it induces $\phi|_{\Sigma^0}$ relative to $\partial^1\Si$, that is $[\phi|_{\Si^0}]_{\partial^1 \Si} = [\phi_{\Lambda^0}]_{\partial^1 \Si}$  
\item it contains the boundary components of $\Si_0$ that intersect with $\calA^1$ and any other boundary component of $\Si^0$ different from $\partial^1$,  that is $B^0=\partial\Si^0 \setminus \partial^1 \Si$. 
\item $\phi|_{B^0} = \phi_{\Lambda^0}|_{B^0}$, (in particular $\phi|_{B^0}$ is an isometry).
\item All the vertices in $B^0$ have valency $3$ (see \Cref{re:v3}). 
\end{itemize}

\end{itemize}

To construct the metric relative spine  $(\G^{\leq i+1},B^{\leq i+1})$  from  $(\G^{\leq i},B^{\leq i})$ we will use the following: 
\begin{itemize} 
\item  parametrizations $\{\eta^{i+1}_{j,k}\}$ of the annuli $\{\calA^{i+1}_{j,k}\}$ as in \Cref{lem:qp} for each $j=1, \ldots, \beta_{i+1}$ and for each $k=1, \ldots, \alpha_j$. For using this lemma we might have to isotope $\phi|_{\calA^{i+1}_{j,k}}$ by an isotopy that preserves the action at the boundary, so in particular, it does not change $\phi$ outside $\calA$. We choose the parametrizations such that $\eta^{i+1}_{j,k}(\MS^1\times\{0\})=\calC'^{i+1}_{j,k}$.

\item  a metric relative spine $(\Lambda^{i+1}_{j,k}, B^{i+1}_{j,k})$ for $\Si^{i+1}_{j,k}$ invariant by $\phi^{\alpha^i_{j}}|_{\Si^{i+1}_{j,k}}$ for each $j=1, \ldots, \beta_{i+1}$, and for each $k=1,..., \alpha_j^i-1$ 
such that 
\begin{itemize}
\item[(i)] $B^{i+1}_{j,k}$ contains all the boundary components of $\Si^{i+1}_{j,k}$ except $\calC''^{i+1}_{j,k}$ (which is invariant by $\phi^{\alpha^i_j}|_{\Si^{i+1}_{j,k}}$). This may include boundary components from the surface $\Si$ not in $\partial^1\Si$, not only boundary components with non-empty intersection with $\calA^{i+1}$, 
\item[(ii)] all the vertices in  $B^{i+1}_{j,k}$ have valency 3 (the important ones are the ones in $\calC'^{i+1}$).
\end{itemize}
Note that condition (1) of the {\em assumptions on $\phi$} (page \pageref{page:assump}) implies that the surface obtained by cutting $\Si^{i+1}_{j,k}$ along $\Lambda^{i+1}_{j,k}$ is a unique cylinder $\widetilde{\Si}^{i+1}_{j,k}$. Let $g_{i+1}|_{\widetilde{\Si}^{i+1}_{j,k}}:\widetilde{\Si}^{i+1}_{j,k}\to \Si^{i+1}_{j,k}$ be the gluing mapping. 
\item mappings  ${r}^{i+1}_{j,k}:$ \begin{equation}\label{eq:r}{r}^{i+1}_{j,k}:\widetilde{\Lambda}^{i+1}_{j,k}\times I\to {\Si}^{i+1}_{j,k},\end{equation}
composition of the gluing map $g_{i+1}|_{\widetilde{\Si}^{i+1}_{j,k}}$ with a  
product structure for the cylinder $\widetilde{\Si}^{i+1}_{j,k}$
\begin{equation}
\label{eq:rtilde}
\widetilde{r}^{i+1}_{j,k}:\widetilde{\Lambda}^{i+1}_{j,k}\times I\to \widetilde{\Si}^{i+1}_{j,k},\end{equation} 
which is invariant by the lifting of $\phi^{\alpha_j}|_{\Si^{i+1}_{j,k}}$ to $\widetilde{\Si}^{i+1}_{j,k}$. 

We choose ${r}^{i+1}_{j,1}(x,1)$ such that:
\begin{itemize}
\item[(iii)] ${r}^{i+1}_{j,k}(x,1)$ is not a vertex of $\Lambda^{i+1}_{j,1}$ whenever $\eta_{j,k}^{i+1}(x,0)$ is a vertex of $\G^{\leq i}$ 
\end{itemize}

\end{itemize}

The reader can now have a look at the final construction of the spine in (a)-(e) in page \pageref{a-e}, in order to have an impression of the final construction.. 

We first explain  more details about how to get 
$(\Lambda^{i+1}_{j,k},B^{i+1}_{j,k})$ and the $r^{i+1}_{j,k}$  satisfying $(i)-(iii)$. Afterwards we will show how we choose the right metric.  


To construct  $\Lambda^{i+1}_{j,k}$, we first find  $\Lambda^{i+1}_{j,1}$ and $r_{j,1}^{i+1}$ and then define $\Lambda^{i+1}_{j,k+1}:=\phi^k(\Lambda^{i+1}_{j,1})$. 

To find the spine $(\Lambda^{i+1}_{j,1}, B^{i+1}_{j,1})$ we consider the quotient map by the action of $\phi^{\alpha_j}$ in $\Si^{i+1}_{j,1}$ as in the proof of \Cref{thm:000} or \ref{thm:rel}. We choose a relative spine of the quotient surface such that: 
\begin{itemize}
\item[-] it  contains the image of all points whose isotropy subgroup by the action of the group generated by $\phi^{\alpha^i_{j}}|_{\Si^{i+1}_{j,k}}$ is non-trivial, 
\item[-] it contains all the boundary components except the image by the quotient map of $\calC''^{i+1}$,
\item[-] it has only vertices of valency 3 along the boundary.
\end{itemize} 
Then we denote by $\Lambda^{i+1}_{j,k}$ its preimage by the quotient map which is a relative spine satisfying (i)-(ii). Moreover, we can lift any regular retraction (or product structure in the cylinder) in the quotient and find a product structure for the cylinder $\widetilde{\Si}^{i+1}_{j,1}$
\begin{equation}\widetilde{r}^{i+1}_{j,1}:\widetilde{\Lambda}^{i+1}_{j,1}\times I\to \widetilde{\Si}^{i+1}_{j,1},\end{equation} that is invariant by the lifting of $\phi^{\alpha_j}|_{\Si^{i+1}_{1,j}}$ to $\widetilde{\Si}^{i+1}_{j,1}$. 
To get (iii) we choose carefully the retraction in the quotient such that $\widetilde{r}^{i+1}_{j,1}(x,1)$ does not correspond to a vertex of $\Lambda^{i+1}_{j,1}$ whenever $\widetilde{r}^{i+1}_{j,1}(x,0)$ is a point preimage either of a vertex of $\G^{\leq i}$ or the image of a vertex of $\G^{\leq i}$ by any power of $\phi$. 

Now we choose the right metric. 

The metric of the graph $\G^{\leq i}$ assigns a metric on $\calC'^{i+1}_{j,k}$ for every $j,k$. We use the natural identification from $\calC'^{i+1}_{j,k}$ to $\calC''^{i+1}_{j,k}$ by $\eta^{i+1}_{j,k}$ (i.e. $\eta^{i+1}_{j,k}(x,0)\mapsto\eta^{i+1}_{j,k}(x,1)$) to put a metric on $\calC''^{i+1}_{j,k}$. 

Note that from the 2 boundary components of $\widetilde{\Si}^{i+1}_{j,1}$, one comes from cutting $\Lambda^{i+1}_{j,1}$ and the other is $ \calC''^{i+1}_{j,1}$. Now, for each $j=1, \ldots, \beta_{i+1}$ we put a metric on $\widetilde{\Lambda}^{i+1}_{j,1}$ by pull back the metric in $ \calC''^{i+1}_{j,1}$ with the mapping given by $\widetilde{r}^{i+1}_{j,1}(x,1)\mapsto \widetilde{r}^{i+1}_{j,1}(x,0)$.  

Let $g_{i+1}|_{\tilde{\Lambda}^{i+1}_{j,1}}:\tilde{\Lambda}^{i+1}_{j,1}\to \Si^{i+1}_{j,1}$ be the restriction of the gluing map. The metric on $\tilde{\Lambda}^{i+1}_{j,1}$ is compatible with the gluing because $\phi^{\alpha_j}|_{\calC''^{i+1}_{j,1}}$ is an isometry and $\phi^{\alpha_j}$ respects retraction lines. 

We define the metric on  $\Lambda^{i+1}_{j,k+1}=\phi^k(\Lambda^{i+1}_{j,1})$ to be the pullback metric on $\Lambda^{i+1}_{j,1}$.

We denote by $(\Lambda^{i+1},B^{i+1})$ the union of the  graphs $(\Lambda^{i+1}_{j,k},B^{i+1}_{j,k})$ for all $j, k$. Note that the metric that we have put on  makes $\phi|_{\Lambda^{i+1}}$ an isometry. 

We build ${\G}^{\leq i+1}$ starting with  $(\G^{\leq i},B^{\leq i})\hookrightarrow  \Si^{\leq i}\cont \Si^{\leq i+1}$ and doing following:
\begin{itemize}\label{a-e}
\item[(a)] We remove $\calC'^{i+1}_{j,k}$ from $\G^{\leq i}$ for every $j$ and $k$. 
\item[(b)] For every edge $e$ in $\G^{\leq i}\setminus \calC'^{i+1}$ containing a vertex in $ \calC'^{i+1}_{j,k}$, if  $L$ is its length in $\G^{\leq i}$, then, we redefine its metric to $L - \epsilon$ (to simplify we take $\epsilon$ smaller than the lengths of every edge).

\item[(c)] We add the embedded segments $\eta^{i+1}_{j,k}(I\times \{x\})$ for all $x$ such that $\eta^{i+1}_{j,k}(0,x)$ is a vertex of $B^{i}$ (that is, a vertex in $\calC'^{i+1}$. We set the length of each of this segments to be $\epsilon/2$.

\item[(d)] We add the embedded segments ${r}^{i+1}_{j,k}(I\times \{x\})$ that concatenate with the ones added in the previous step. We set the length of each of this segments to be $\epsilon/2$.

\item[(e)] We add ${\Lambda}^{i+1}$ with the metric we  defined previously. 
\end{itemize}
We set $B^{\leq i+1}:=B^{\leq i}\cup B^{i+1}\setminus\calC'^{i+1}$. 

This is obviously a relative metric spine for $\Si^{\leq i+1}$. 

We observe that $(\G^{\leq i+1}_{{\Lambda}^{i+1}},B^{\leq i+1}_{{\Lambda}^{i+1}})$ is isometric to $(\G^{\leq i},B^{\leq i})$. This will allow us to codify all the information of the $(\G^{\leq i},B^{\leq i})$ in the metric spine $(\G^{\leq d}, B^{\leq d})$ (for $d$ maximal). We will  define a  filtration to keep the recursive steps in the construction we  have just explained.

\textbf{Recovering $[\phi]_{\partial^1}$ from the metric relative spines for $(\Si^{\leq i},B^{\leq i})$}.

Now we explain in which sense the sequence of metric spines $(\Gamma^{\leq i},B^{\leq i})$ that we have constructed, together with some extra numerical information, codify the automorphisms
$$[\phi_{\Si^{\leq i}}]_{B^{\leq i},\phi|_{B^{\leq i}}}$$
for all $i$.

For $i=0$ it is immediate and we don't need extra information because $(\G^{\leq 0},B^{\leq 0})$ is a $\pi$-\tat graph codifying it in the sense of \Cref{thm:rel}. 

Assume that we know how to recover the automorphisms up to a certain $i$, let us see how to recover the automorphism for $i+1$. 

Firstly, we make some observations about the original $\phi$ in quasi-canonical form as in (\ref{re:stan}). 

We define the homeomorphism $\bar{\phi}_i$ as the homeomorphism of the $(\Si^{\leq i+1})_{\Lambda^{i+1}}$  that coincides with ${\phi}|_{\Si^{\leq i}}$ in ${\Si^{\leq i}}$ and that extends \emph{trivially} to the remaining cylinders as in \Cref{re:g2} using the parametrizations $\eta_{j,k}^{i+1}$ and $\tilde{r}^{i+1}_{j,k}$ for every $j,k$ successively.

Recall \Cref{not:D_an} and for every $j,k$ consider 
the Dehn twists $$\eta_{j,k}^{i+1}\circ\calD_{-s_j / 
\alpha_j,0}\circ(\eta_{j,k}^{i+1})^{-1}$$ along the annuli $\calA^{i+1}_{j,k}$  
with  $s_j$ the screw number of $\phi$ on at $\calA^{i+1}_{j,k}$. We extend the 
composition of all this Dehn twists to $\Si^{\leq i+1}_{\Lambda^{i+1}}$ in the 
following way: we extend by the identity to $\Sigma^{\leq i}$ and we extend to 
the cylinders $\Si^{i+1}_{\Lambda^{i+1}}$ as in \Cref{re:g2} using the 
parametrizations $\tilde{r}^{i+1}_{j,k}$. We denote by 
$$\calD_i:\Si^{\leq i+1}_{\Lambda^{i+1}}\to\Si^{\leq i+1}_{\Lambda^{i+1}}$$
the extension that we just have constructed.
  
{Observe that $\calD_{i}\circ {\bar{\phi}}_i$ is compatible with the gluing $g_{i+1}$ because its restriction to  $\Si^{\leq i}\cup \calA^{i+1}$ coincides with $\phi|_{\Si^{\leq i}\cup \calA^{i+1}}$ up to isotopy fixing $\calC''^{i+1}$.}     
Then, it induces a mapping in $\Si^{\leq i+1}$ that coincides with $\phi|_{\Si^{\leq i+1}}$ up to isotopy fixing the boundary.

Now, we come back to the graph $(\G^{\leq i+1},B^{\leq i+1})$. Assume we know $\phi|_{\Si^{\leq i}}$ up to isotopy fixing the action at the boundary $B^{\leq i}$. We extend trivially $\phi|_{\Si^{\leq i}}$ using parametrizations of the collars of $\Si^{\leq i+1}_{\Lambda^{i+1}}\setminus\Si^{\leq i}$ for which the segments of $\G^{\leq i+1}$ inside are retraction lines of the parametrizations (use \Cref{re:g2}). We call it $\widetilde{\phi}_i$.  It is a homeomorphism of $\Si^{\leq i+1}|_{\Lambda^{i+1}}$. 
Let $\ell_{j}^{i+1}$ be the length of $\tilde{\Lambda}^{i+1}_{j,k}$ for any $k$, that is, it coincides with the original $length(\calC''^{i+1}_{j,1})$. Define 

\begin{equation}\label{eq:delta}\delta^{i+1}_j:=-s_j/\alpha_j \cdot \ell_{j}^{i+1}
\end{equation}

with $s_j$ the screw number of $\phi$ on $\calA^{i+1}_{j,k}$.
Let $\calD_{\delta_{i+1}}$ be the composition of the {boundary} Dehn twists of length {$\delta^{i+1}_j$}  along each $\tilde{\Lambda}^{i+1}_{j,k}$ (see \Cref{def:boundary_dehn}). 
Then, it is clear, by the previous observations about $\phi$, that $[\widetilde{\phi}_{i}]_{\partial,\widetilde{\phi}_{i}|_\partial}$ is compatible with the gluing $g_{i+1}$ and that the homeomorphism it induces in $\Si^{\leq i+1}$ coincides with $\phi|_{\Si^{\leq i+1}}$ up to isotopy fixing the {\em action at the} boundary.

Below we see the diagram which shows how $[\phi]_{\partial^1}$ is obtained from the metric graphs $(\G^{\leq i},B^{\leq i})$ and {the numbers $\delta^i_j$}.

\begin{equation}
\begin{tikzcd}\label{diag:mixed_tat}
\Si^{\leq 1}_{\Lambda^{1}} \arrow{d}{g_1} \arrow[rightarrow]{rr}{\widetilde{\phi}_{0}}  &  & \Si^{\leq 1}_{\Lambda^{1}} \arrow[rightarrow]{r}{\calD_{\delta_1}} &   \Si^{\leq 1}_{\Lambda^{1}} \arrow{d}{g_1}  &  &  & &\\
\Si^{\leq 2}_{\Lambda^{2}} \arrow{d}{g_2} \arrow{rrr}{ \widetilde{ \phi}_{1}}&  &  & \Si^{\leq 2}_{\Lambda^{2}}\arrow{r}{\calD_{\delta_2}} & \Si^{\leq 2}_{\Lambda^{2}}  \arrow{d}{g_2} &   &&&\\
\vdots & \vdots  & \vdots &  \vdots & \vdots &  &    & \\
\Si^{\leq d}_{\Lambda^{d}} \arrow{d}{g_d} \arrow{rrrrrr}{ \widetilde{ \phi}_{d-1}}& &&&&&\Si^{\leq d}_{\Lambda^{d}}  \arrow{r}{\calD_{\delta_d}}&\Si^{\leq d}_{\Lambda^{d}}  \arrow{d}{g_d}\\
\Si \arrow{rrrrrrr}{\phi=\widetilde{ \phi}_{d}} &&&&&&& \Si\\
\end{tikzcd}\end{equation}

We have obtained the following: 
\begin{proposition}
\label{prop:tat00} 
The {collection} of metric spines $(\G^{\leq i},B^{\leq i})$ for $i=0,...,d$ together with the numbers $\delta^i_j$ obtained in (\ref{eq:delta}) determine $[\phi]_{\partial\Si,\phi|_{\partial\Si}}$.
\end{proposition}

\textbf{Filtration on $\G^{d}$}
We can codify the information of the {collection} the metric relative spines $(\G^{\leq i},B^{\leq i})$ by a filtration on the last spine $(\G^{\leq d},B^{\leq d})$ as follows. Define 
  $$\G^i:=\bigcup_{\ell\geq i} (\Lambda^\ell \setminus \calC'^{\ell+1}) \cup \bigcup_{\ell \geq i\\ j,k} \eta_{j,k}^{\ell+1}(I\times \{x\})\cup \bigcup_{\ell \geq i\\  j,k} r_{j,k}^{\ell+1}(I\times \{x\}),$$
  
  $$A^{i}:= B^{\leq d}\cap \G^i$$
 where $x$ runs exactly as in the construction in the steps (c) and (d). 
 
In this way we obtain a filtered relative spine
$$(\G^{\leq d},B^{\leq d})=(\G^0,A^0)\supset (\G^1,A^1)\supset ...\supset (\G^d,A^d).$$

\begin{proposition}
\label{prop:tat0} 
The metric filtered graph  defined above, together with the numbers $\delta^i_j$ in (\ref{eq:delta}), determine $[\phi]_{\partial\Si,\phi|_{\partial\Si}}$.
\end{proposition}
\begin{proof}
This follows from \Cref{prop:tat00} since $\G^{0}_{\G^{i+1}}$ is isometric to $\G^{\leq i}$. 
\end{proof}

The properties of this filtered metric relative spine and the numbers $\delta_i$, which can be summarized in the diagram \ref{diag:mixed_tat}, can be restated in terms of mixed \tat graphs that we introduce in the next section.

\subsection{Mixed \tat graphs.}\label{sec:mixed}

Now we introduce the definition of mixed \tat graphs and twists, inspired by the constructions of the previous section.

Let $(\G^\bullet, A^{\bullet})$ be a decreasing filtration on a connected relative metric ribbon graph $(\G, A)$. That is
$$(\G,A)=(\G^0, A^0) \supset (\G^1, A^1) \supset \cdots \supset (\G^d, A^d)$$
where $\supset$ between pairs means $\G^i \supset \G^{i+1}$ and $A^i \supset A^{i+1}$, and where $(\G^i, A^i)$ is a (possibly disconnected) relative metric ribbon graph for each $i=0,\ldots, d$. We say that $d$ is the depth of the filtration $\G^\bullet$. We assume each $\G^i$ does not have univalent vertices and is a subgraph of $\G$ in the usual terminology in Graph Theory. We observe that since each $(\G^i, A^i)$ is a relative metric ribbon graph, we have that $A^i \setminus A^{i+1}$ is a disjoint union of connected components homeomorphic to $\MS^1$.

For each $i=0, \ldots, d$, let $$\delta_i: \G^{i} \rightarrow \R_{\geq 0}$$ be a locally constant map (so it is a map constant on each connected component). We put the restriction that $\delta_0(\G^0)>0$. We denote the collection of all these maps by $\delta_\bullet$.

Given $p \in \G$, we define $c_p$ as the largest natural number such that $p\in \G^{c_p}$.

\begin{definition}[Mixed safe walk]\label{def:mixed_tat} Let $(\G^\bullet, A^{\bullet})$ be a filtered relative metric ribbon graph. 
Let $p \in \G\setminus A \setminus v(\G)$. We define a mixed safe walk $\gamma_p$ starting at $p$ as  a concatenation of paths defined iteratively by the following properties
\begin{itemize}
\item[i)] $\gamma_p^0$ is a safe walk of length $\delta_0(p)$ starting at $p^\gamma_0:=p$. Let $p^\gamma_1:=\gamma^0(\delta_0)$ be its endpoint.
\item[ii)] Suppose that $\gamma_p^{i-1}$ is defined and let $p^\gamma_i$ be its endpoint. 
\begin{itemize}
\item If $i > c_p$ or $p^\gamma_{i} \notin \G^{i}$ we stop the algorithm. 
\item If $i\leq c_p$ and $p^\gamma_{i} \in \G^{i}$ then define $\gamma_p^{i}:[0,\delta_i(p_i)] \rightarrow \G^i$ to be a safe walk of length $\delta_i(p^\gamma_i)$ starting at $p^\gamma_i$ and going in the same direction as $\gamma_p^{i-1}$.
\end{itemize} 
\item[iii)] Repeat step $ii)$ until algorithm stops.
\end{itemize}

Finally, define $\gamma_p:=\gamma_p^k \star \cdots \star \gamma_p^0$, that is, the mixed safe walk starting at $p$ is the concatenation of all the safe walks defined in the inductive process above. 
\end{definition}

As in the pure case, there are two safe walks starting at each point on $\G \setminus (A \cup v(\G))$. We denote them by $\gamma_p$ and $\omega_p$.

\begin{definition}[Boundary mixed safe walk]\label{def:boundary_mixed_tat}
Let $(\G^\bullet, A^{\bullet})$ be a filtered relative metric ribbon graph and let $p \in A$. We define a boundary mixed safe walk $b_p$ starting at $p$ as  a concatenation of a collection of paths defined iteratively by the following properties
\begin{itemize}
\item[i)] $b_{p_0}^0$ is a boundary safe walk of length $\delta_0(p)$ starting at $p_0:=p$ and going in the direction indicated by $A$ (as in the relative \tat case). Let $p_1:=b_p^0(\delta_0)$ be its endpoint.
\item[ii)] Suppose that $b_{p_{i-1}}^{i-1}$ is defined and let $p_i$ be its endpoint. 
\begin{itemize}
\item If $i >c_p$ or $p_{i} \notin \G^{i}$ we stop the algorithm. 
\item If $i\leq c(p)$ and $p_{i} \in \G^{i}$ then define $b_{p_i}^{i}:[0,\delta_i(p_i)] \rightarrow \G^i$ to be a safe walk of length $\delta_i(p_i)$ starting at $p_i$ and going in the same direction as $b_{p_{i-1}}^{i-1}$.
\end{itemize} 
\item[iii)] Repeat step $ii)$ until algorithm stops.
\end{itemize}

Finally, define $b_p:=b_{p_k}^k \star \cdots \star b_{p_0}^0$, that is, the boundary mixed safe walk starting at $p$ is the concatenation of all the safe walks defined in the inductive process. 
\end{definition}

\begin{notation}
We call the number $k$ in \Cref{def:mixed_tat} (resp. \Cref{def:boundary_mixed_tat}), the \textit{order} of the mixed safe walk (resp. boundary mixed safe walk) and denote it by $o(\gamma_p)$ (resp. $o(b_p)$). 

We denote by  $l(\gamma_p)$  the \textit{length} of the mixed safe walk $\gamma_p$ which is the sum $\sum_{j=0}^{o(\gamma_p)} \delta_j(p^\gamma_j)$ of the lengths of all the walks involved. We consider the analogous definition $l(b_p)$.

As in the pure case, two mixed safe walks starting at $p \in \G \setminus v(\G)$ exist. We denote by $\omega_p$ the mixed safe walk that starts at $p$ but in the opposite direction to the starting direction of $\gamma_p$.

Observe that since the safe walk $b_{p_0}^0$ is completely determined by $p$, for a point in $A$ there exists only one boundary safe walk.
\end{notation}

Now we define the relative mixed \tat property.

\begin{definition}[Relative mixed \tat property] \label{def:relative_mix_sim}
Let $(\G^\bullet, A^{\bullet})$ be a filtered relative metric ribbon graph and let $\delta_\bullet$ be a set of locally constant mappings $\delta_k: \G^k \rightarrow \R_{\geq 0}$. We say that  $(\G^{\bullet}, A^\bullet, \delta_\bullet)$ satisfies the relative mixed \tat property or that it is a relative mixed \tat graph if for every $p \in \G-(v(\G) \cup A)$

\begin{itemize}
\item[I)] The endpoints of $\gamma_p$ and $\omega_p$ coincide.
\item[II)] $c_{\gamma_p(l({\gamma_p}))}=c_p$
\end{itemize}
and for every $p \in A$, we have that  
\begin{itemize}
\item[III)] $b_p(l({b_p})) \in A^{c_p}$
\end{itemize} 
\end{definition}

As a consequence of the two previous definitions we have:
\begin{lemma} \label{lem:mix_prop}
Let $(\G^{\bullet}, A^\bullet,\delta_\bullet)$ be a mixed relative \tat graph, then

\begin{itemize}
\item[a)] $o(\omega_p)=o(\gamma_p)= c_p$
\item[b)]$l(\gamma_p) = l(\omega_p)$ for every $p \in \G \backslash v(\G)$.
\end{itemize}
\end{lemma}

\begin{proof}
$a)$ By \Cref{def:mixed_tat} ii), we have that $o(\gamma_p) \leq c_p$ for all $p \in \G\setminus v(\G)$. Suppose that for some $p$, we have that $o(\gamma_p) = k < c_p$. This means, that while constructing the mixed safe walk we stopped after constructing the path $\gamma_p^k$ either because $k> c_p$ which contradicts the supposition, or because the endpoint $p_k^\gamma$ of $\gamma_p^k$ is not in $\G^{k}$ which contradicts that $c_p = c_{\gamma_p(l(\gamma_p))}$. This proves the equality $o(\gamma_p)=c_p$. In order to prove $o(\omega_p)=c_p$ use the equality $\gamma_p(l(\gamma_p))=\omega_p(l(\omega_p))$ and repeat the same argument.

$b)$ Let $q$ be the endpoint of $\gamma_p$ and $\omega_p$. Since the image of the safe walks $\gamma_p^{c_p}$ and $\omega_p^{c_p}$ lies on the same connected component of $\G^{c_p}$ we have that their starting points $p^\gamma_{c_p}$ and $p^{\omega}_{c_p}$ also lie on that same connected component. Therefore $\delta_{c_p}(p^\gamma_{c_p}) = \delta_{c_p}(p^\omega_{c_p})$. 

Suppose now that $p^\gamma_i$ and $p^\omega_i$ lie on the same connected component of $\G^i$ (and so $\delta_i(p^\gamma_i) = \delta_i(p^{\omega}_i)$). Then the image of the safe walks $\gamma_p^{i-1}$ and $\omega_p^{i-1}$ lies on the same connected component of $\G^{i-1}$ and we have that their starting points $p^\gamma_{i-1}$ and $p^{\omega}_{i-1}$ also lie on that same connected component. So $\delta_{i-1}(p^\gamma_{i-1}) = \delta_{i-1}(p^{\omega}_{i-1})$. 

We conclude that $\delta_j(p^\gamma_j) = \delta_j(p^{\omega}_j)$ for all $j=0, \ldots, d$ which concludes the proof.
\end{proof}
\begin{remark} Note that for mixed \tat graphs it is not true that $p\mapsto \gamma_p(\delta(p))$ gives a continuous mapping from $\G$ to $\G$. 
\end{remark}

The reader can check that the filtered graph associated with the pseudo-periodic homeomorphism $\phi$ in the previous section is a  mixed relative \tat graph. 
\subsection{Mixed \tat twists.} \label{sec:mixed_homeo}
Let $(\G^{\bullet}, A^\bullet,\delta_\bullet)$  be a a mixed \tat graph  and let 
$(\Sigma,\G,A)$ be a thickening surface of $(\G,A)$.  Let $\partial^1 \Si$ be the union of the boundary components of $\Si$ not contained in $A$. In this section we define a pseudo-periodic homeomorphism $\phi_{(\Sigma,\G^{\bullet},\delta)}$ of $\Si$ associated to $(\G^{\bullet}, A^\bullet,\delta_\bullet)$. This homeomorphism is well defined up to isotopy fixing $\partial^1 \Si$ and relative to the action on $A$.

For the sake of simplicity in notation we assume that $A^\bullet=\emptyset$ during the construction. The general case is analogous.

\begin{notation}

Let $g_{\G^i}:\Si_{\G^i}\to \Si$ be the gluing map as in \Cref{not:cut}. Let $\G_{\G^i}$ be the preimage of $\G$ by $g_{\G^i}$. We also denote  by $g_{\G^i}$ its restriction $g_{\G^i}:\G_{\G^i}\to \G$. 
The union of the boundary components of $\Si_{\G^i}$ that come from $\G^{i}$ is denoted by $\widetilde\G^{i}$. Observe that a single connected component of $\G^i$ might produce more than one boundary component in $\Si_{\G^i}$.

It's clear that $g_{\G^i}$ factorizes as follows: 
$$\Si_{\G^i}\to \Si_{\G^{i+1}}\to...\to \Si.$$
We denote these mappings by $g_j:\Si_{\G^j}\to \Si_{\G^{j+1}}$ for $j=0,...,d-1$ and also their restrictions $g_j:\G_{\G^j}\to \G_{\G^{j+1}}$.  

\end{notation}

\begin{remark}\label{rem:relative_tat_f}
Observe that by \Cref{def:relative_mix_sim}, each connected component of the relative metric ribbon graph $(\G_{\G^{1}},\TG^1)\hookrightarrow \Si_{\G^1}$ has the relative \tat property for safe walks of length $\delta_0(\G)$.
\end{remark}

Let $\phi_{\G,0}: \Si_{\G^1} \rightarrow \Si_{\G^1}$ be the induced \tat twist fixing each boundary component that is not in $\TG^1$ as in  \Cref{def:phi_Gtau}. It is the homeomorphism induced by the relative \tat property of each connected component of $(\G_{\G^1}, \TG^1)$ for some choice of product structures on $(\Si_{\G^1})_{\G_{\G^1}}$. 

Also according to \Cref{def:phi_Gtau}, observe that since we do not specify anything, we assume that the sign $\iota$ is constant $+1$.

Now we continue to define inductively the homeomorphism $\phi_{(\Sigma,\G^{\bullet},\delta)}$. 

\begin{notation}\label{not:boundary_dehn}
Let $$\calD_{\delta_i}: \Si_{\G^i} \rightarrow \Si_{\G^i}$$ be the homeomorphism consisting of the composition of all the boundary Dehn twists $\calD_{\delta_i(g_{\G^i}(C))}(C)$ for all components in $\TG^i$. Recall \Cref{def:boundary_dehn}.
\end{notation}

\begin{lemma}
The homeomorphism $$\widetilde{\phi}_{\G,1}:=\calD_{\delta_1} \circ \phi_{\G,0}:\Si_{\G^1} \rightarrow \Si_{\G^1}$$
is compatible with the gluing $g_1:\Si_{\G^1}\to \Si_{\G^2}$.
\end{lemma}

\begin{proof}
We use the notation introduced in \Cref{def:mixed_tat}.

 Since $g_1$ only identifies points in $\widetilde{\G}^1$, we  must show that if $x$, $y$ are different points in $\widetilde\G^{1}$ such that $g_1(x)=g_1(y)\in \G^1$, then $g_1(\widetilde{\phi}_{\G,1}(x))=g_1(\widetilde{\phi}_{\G,1}(y))$. 

So let $x,y \in \TG^1$ be such that $g_1(x)=g_1(y) = p$. In particular, $c_p=1$ and we have that  $\gamma_p=\gamma_p^1 \star \gamma^0_{p}$ and $\omega_p=\omega_p^1 \star \omega_{p}^0$ by \Cref{lem:mix_prop} $a)$. So the mixed safe walks end in a connected component of $\G^1$. Denote by $\hat{p}$ their endpoint. By \Cref{def:relative_mix_sim} $ii)$ we have that $c_{\hat{p}} = 1$.

First observe that $\phi_{\G,0}(x)=b_x(\delta_0(p))$ where $b_x:[0,\delta_0(p)] \rightarrow \G_{\G^1}$ is the  boundary safe walk of length $\delta_0(p)$ given by the relative \tat structure on $(\G_{\G^1}, \TG^1)$. Analogously $\phi_{\G,0}(y)=b_y(\delta_0(y))$. So we have $$g_1(\phi_{\G,0}(x)) = \gamma_{p}(\delta_0(p))$$ and $$g_1(\phi_{\G,0}(y)) = \omega_{p}(\delta_0(p))$$. 

It is clear that $(\calD_{\delta_1}(\phi_{\G,0}(x))) = \widetilde{\gamma}^0_p(\delta_1(\phi_{\G,0}(x)) + \delta_0(p))$
 with $\widetilde{\gamma}_p$ the safe walk in $\widetilde{\G}^1$. It is also clear that $\widetilde{\gamma}_p$ is the actual lifting of $\gamma_p$ along $\G$. Then,  
$$g_1(\calD_{\delta_1}\circ\phi_{\G,0}(x)) = \gamma_p(\delta_1(\phi_{\G,0}(x)) + \delta_0(p)) =  \gamma_p(\delta_1(\gamma^0_p(\delta_0(p))) + \delta_0(p))= \gamma_p(l(\gamma_p))$$
and analogously $g_1(\calD_{\delta_1} \circ\phi_{\G,0}(y)) = \omega_p(l(\omega_p))$.

By property $i)$ of a mixed \tat graph, we can conclude. 
\end{proof}
Now, we consider the homeomorphism induced by $\widetilde{\phi}_{\G,1}$ and we denote it by $$\phi_{\G,1}: \Si_{\G^{2}} \rightarrow \Si_{\G^{2}}.$$

The same argument applies inductively to prove that each map
\begin{equation}\label{eq:notation_mixedpartial}
\widetilde{\phi}_{\G,i}:=\calD_{\delta_i} \circ \phi_{\G,i-1}: \Si_{\G^i} \rightarrow \Si_{\G^i}
\end{equation}
is compatible with the gluing $g_i$ and hence it induces a homeomorphism \begin{equation}\label{eq:notation_mixedpartial22}\phi_{\G,i}: \Si_{\G^{i+1}} \rightarrow \Si_{\G^{i+1}}.\end{equation}

In the end we get a map 
\begin{equation}\label{eq:mix_phi_G}\phi_{(\Sigma,\G^{\bullet},\delta)}:=\phi_{\G,d}: \Si \rightarrow \Si\end{equation}
 which we call the \emph{mixed \tat twist} induced by $(\G^{\bullet}, \delta_\bullet)$ and the chosen embedding $\G\hookrightarrow \Sigma$.

\begin{notation}\label{not:ext_mixed_tat}
We can extend the notation introduced before by defining  $\phi_{\G, -1}:= id$ and $\widetilde{\phi}_{\G,0}:= \calD_{\delta_0} \circ \phi_{\G,-1} = \calD_{\delta_0}$. 
Then we can restate \Cref{rem:relative_tat_f} by saying that $\widetilde{\phi}_{\G,0}$ is compatible with the gluing $g_0$ and induces the homeomorphism $\phi_{\G,0}$. 
\end{notation}

\begin{remark}\label{rem:equiv_mixed_tat}

After the description of the construction of the mixed \tat twist above and the diagram \ref{diag:mixed_tat} and \ref{diag:mixed_tat_2}, we observe that satisfying $I)$ and $II)$ of the mixed \tat property in  \Cref{def:relative_mix_sim} is equivalent to satisfying:
\begin{itemize}
\item[I')] For all $i=0, \dots, d-1$, the homeomorphism $\widetilde \phi_{\G,i}=\calD_{\delta_i}\circ\phi_{\G,i-1}$ is compatible with the gluing $g_i$, that is, $$g_i(x) =g_i(y) \Rightarrow g_i(\widetilde{\phi}_{\G,i}(x)) = g_i(\widetilde{\phi}_{\G,i}(y)).$$

Below we see the diagram which shows the construction of $\phi_\G$.

\begin{equation}
\begin{tikzcd}\label{diag:mixed_tat_2}
\Si_{\G^{0}} \arrow[rightarrow]{r}{ \phi_{\G,-1}} \arrow{d}{g_0} & \Si_{\G^{0}} \arrow[rightarrow]{r}{\calD_{\delta_0}} & \Si_{\G^{0}}  \arrow{d}{g_0} & & & &   &\\
\Si_{\G^{1}} \arrow{d}{g_1} \arrow[rightarrow]{rr}{ \phi_{\G,0}}  &  & \Si_{\G^{1}} \arrow[rightarrow]{r}{\calD_{\delta_1}} &   \Si_{\G^{1}} \arrow{d}{g_1}  &  &  & &\\
 \Si_{\G^2} \arrow{d}{g_2} \arrow{rrr}{ {\phi}_{\G,1}}&  &  &  \Si_{\G^{2}} \arrow{r}{\calD_{\delta_2}} & \Si_{\G^{2}}  \arrow{d}{g_2} &   &&&\\
\vdots & \vdots  & \vdots &  \vdots & \vdots &  &    & \\
\Si_{\G^d} \arrow{d}{g_d} \arrow{rrrrrr}{ {\phi}_{\G,d-1}}& &&&&&\Si_{\G^d} \arrow{r}{\calD_{\delta_d}}& \Si_{\G^d} \arrow{d}{g_d}\\
\Si \arrow{rrrrrrr}{\phi_{\G}=\phi_{\G,d}} &&&&&&& \Si\\
\end{tikzcd}\end{equation}
\end{itemize}

\end{remark}

We prove the following:
\begin{theorem}\label{theo:tat_qp}
The homeomorphism $\phi_{\G,i}$ is
pseudo-periodic for all $i=0,...,d$. In particular, $\phi_\G$ is pseudo-periodic.
\end{theorem}

\begin{proof}
The mapping $\phi_{\G,0}$ is periodic. 
Assume $\phi_{\G,i-1}$ is pseudo-periodic. Let's see so is $\phi_{\G,i}$.
Choose a collar neighbourhood $\calU^i$ of $\widetilde{\G}^i$, invariant by $\phi_{\G,i-1}$. We denote by $\calU^i_{j,1}$, ..., $\calU^i_{j,\alpha_j}$ any set of its connected annular components permuted by $\phi_{\G,i-1}$ such that $\phi_{\G,i-1}(\calU^i_{j,k})=\calU^i_{j,k+1}$.
Similarly to \Cref{re:g1} we can assume that, up to  isotopy fixing the action on the boundary, the homeomorphism  $\phi_{\G,i-1}$  satisfies that for some parametrizations $\eta^i_{j,k}:\MS^1 \times I\to \calU^i_{j,k}$ we have that $$p_2\circ (\eta^i_{j,k+1})^{-1}\circ\phi_{\G,i-1}|_{\calU^i_{j,k}}\circ \eta^i_{j,k}(x,t)= p_2\circ(\eta^i_{j,k+1})^{-1}\circ\phi_{\G,i-1}|_{\calU^i_{j,k}}\circ \eta^i_{j,k}(x,t')$$ for every $t,t'\in I$ where $p_2:\MS^1 \times I\to \MS^1$ is the canonical projection. 

Now we consider $\widetilde{\phi}_{\G,i}:=\calD_{\delta_{i}}\circ\phi_{\G,i-1}$ as in (\ref{eq:notation_mixedpartial}). 
The curves $\calC^i_{j,k}=\eta^i_{j,k}(\MS^1 \times \{0\})$ are invariant by $\widetilde{\phi}_{\G,i}$. These curves separate $\Si_{\G^i}$ in two pieces: $\calU^i$ and its complementary that we call $\calB$. After quotienting by $g_{i}$,  we get a $\phi_{\G,i}$-invariant piece that is $g_i(\calB)\approx \calB$ and another one that is $g_i(\calU^i)$. The restriction of $\phi_{\G,i}$ to $g_i(\calB)$ is conjugate to $\phi_{\G,i-1}$ and then pseudo-periodic. The restriction to $g_i(\calU^i)$ has an invariant spine that is $g_i(\widetilde{\G}^i)$ and then, by  \Cref{lem:invariant}, it is boundary-free isotopic to a periodic one. Then, we have seen that $\phi_{\G,i}$ is pseudo-periodic.    
\end{proof}

\begin{remark} The screw number associated to the orbit
of an invariant annuli at the curve $\calC^i_{j,k}$ that appears in the previous proof is 
\begin{equation}\label{eq:sc}s(\calC^i_{j,k})=-\sum_k \delta^1_{j,k} / l(\TG^1_{j,1}).\end{equation}

We check this by giving a more elaborated construction that the one in the proof of \Cref{theo:tat_qp} from which the computation of the screw number follows easily. 

 

Let $\TG^i_{j,1}, \ldots, \TG^i_{j,\alpha_j}$ be a set of boundary components of $\Si_{\G^i}$ contained in $\TG^i$ and cyclically permuted by $\phi_{\G,i-1}$. Let $\calA^i_{j,k}$ annular neighbourhoods of them as in the previous proof such that $\phi_{\G,i-1}(\calU^i_{j,k})=\calU^i_{j,k+1}$. 

Since $\phi_{\G,i-1}^\alpha|_{\calU^i_{j,1}}$ is periodic, we can choose coordinates/parametrization $r_{j,1}^i$ from $\MS^1\times[0,1]$ (with $\MS^1\approx\R/\Z$ of total length 1) to $\calU^i_{j,1}$ with respect to which $\phi_{\G,i-1}^\alpha|_{\calU^i_{j,1}}$ is a rotation of the annulus, that is 
$$(r^i_{j,1})^{-1}\circ \phi^{\alpha_j}_{\G,i-1}\circ r^i_{j, 1}(x,t)=(x+\tau_{i-1},t),$$ where $\tau_{i-1}$ is the rotation number  $rot(\phi_{\G,i-1}^{\alpha_j}|_{\TG^i_{j,i}})$  (note that $\tau_{i-1}$ depends on the considered orbit of annuli, not only on $i$). 
Assume $\TG^i_{j,k}=r^i_{j,k}(\MS^1\times\{1\})$. 
Let $l(\TG^i_{j,k})$ be the length of $\TG^i_{j,k}$. We assume without loss of generality that $r^i_{j,k}|_{\MS^1\times\{1\}}$ is a homotethy of ratio $l(\TG^i_{j,k})$ onto $\TG^i_{j,k}$.

Define $r^i_{j,k}:=\phi_{\G,0}^{k-1} \circ r^i_{j,1}$ for $k=2, \ldots, \alpha_j-1$. 

Define now $\calA^i_{j,k}:= r^i_{j,k}(\MS^1 \times [0, 1/2])$ and $\calC^i_{j,k}:= r^i_{j,k}(\MS^1 \times \{1/4\})$ (note that this system of curves $\calC^i_{j,k}$ is different from that of the previous proof but isotopic to it). We choose a representative of the boundary Dehn twists $\calD_{\delta_i}$ (defined up to isotopy fixing the boundary) performing \textit{all the twisting} in the annuli $\calA^i_{j,k}$. More precisely we can assume $\calD_{\delta_i}$ is expressed on $\calU^i_{j,k}$ in the coordinates $r^i_{j,k}$ as follows:
\begin{equation}\nonumber
(r^i_{j,k+1})^{-1}\circ \calD_{\delta_i}\circ r^i_{j, k}(x,t) := 
\left\{ \begin{array}{lll} 
(x+2\cdot t\cdot \delta^i_{j,k}/l(\TG^i_{j,k}), t) &\mbox{if}& 0 \leq t \leq 1/2 \vspace*{0.2cm} \\ 
(x + \delta^i_{j,k}/l(\TG^i_{j,k}), t)& \mbox{if} & 1/2 \leq t \leq 1
\end{array}
\right.
\end{equation}
 where we denote by $\delta^i_{j,k}:= \delta_i(g_i(\TG^i_{j,k}))$ with $g_i:\Si_{\G^i} \rightarrow \Si_{\G^{i+1}}$ the gluing function.

Then, we have that $\tilde{\phi}_{\G,i} :=\calD_{\delta_i} \circ \phi_{\G,i-1}$ satisfies that 

\begin{equation}\label{eq:eq}
(r^i_{j,1})^{-1}\circ \tilde\phi_{\G,i}^{\alpha_j}\circ r^i_{j, 1}(x,t) =\left\{ \begin{array}{lll} 
\left(x + \tau_{i-1}+2\cdot t\cdot \sum_k\delta^i_{j,k}/l(\TG^i_{j,1}) , t \right) &\mbox{if}& 0 \leq t \leq 1/2 \vspace*{0.2cm} \\ 
\left(x + \tau_{i-1}+ \sum_k\delta^i_{j,k}/l(\TG^i_{j,1}) , t\right)& \mbox{if} & 1/2 \leq t \leq 1
\end{array}
\right.
\end{equation}
where $\tau_{i-1}=rot(\phi_{\G,i-1}^{\alpha_j}|_{\TG^i_{j,i}})$ is the rotation number of $\phi_{\G,i-1}^{\alpha_j}$ at  $\TG^i_{j,1}$. Recall that this rotation number is computed orienting $\TG^{i}_{j,1}$ with the orientation induced by relative safe walks, that is, the opposite as the  orientation that it inherits as boundary component of $\Si_{\G^i}$. 

From the expression in (\ref{eq:eq}) we see that the restriction of $\tilde{\phi}^{\alpha_j}_{\G,i}$ to $\calA^i_{j,k}$ is conjugated to $\calD_{s,\tau_{i-1}}$ (note that we have to reparametrize $t=t/2$ in order to get the expression of $\calD_{s,\tau_{i-1}}$ which is defined in $\MS^1\times [0,1]$) where $s$ is the opposite of the expression in (\ref{eq:sc}) . Applying \Cref{def:screw_number}, we get that the screw number at $\calA^i_{j,k}$ are given exactly by by the expression (\ref{eq:sc}).

By \Cref{def:screw_number} (see also \Cref{re:screw_number}) this computes the screw number at $\calC^i_{j,k}$ since the mapping $\widetilde{\phi}_{\G,i}$, after quotienting by $g_i$, goes down to a  homeomorphism that is periodic restricted to $g_i(\calU^i_j\setminus \calA_j^i)$ where $\calU^i_j=\cup_k\calU^i_{j,k}$ and $\calA^i_j=\cup_k\calA^i_{j,k}$,  and to a pseudo-periodic homeomorphism of $\Sigma_{\G^i}\setminus \calU^i\cont\Sigma_{\G^{i+1}}$ in a quasi-canonical form (see \Cref{re:stan}).

\end{remark}

\begin{remark}\label{re:conj_mixed}
Given a relative mixed \tat graph $(\G^{\bullet}, A^{\bullet}, \delta_{\bullet})$ and given two different 
embeddings of $(\G,A)$ into the surface $\Sigma$, 
the induced mixed \tat twists are conjugated by the same homeomorphism that relates the two  embeddings.
\end{remark}

%
%

\subsection{A realization theorem. }\label{sec:final}
It is straightforward that the filtered metric graph described at the end of \Cref{sec:restricted} is a mixed relative \tat graph with the given $\delta$'s. The construction given in that section proves the following: 

\begin{theorem}\label{thm:mixed_model}
Let $\phi$ be a pseudo-periodic homeomorphism 
satisfying assumptions (1)-(3) at the beginning of \Cref{sec:restricted}. Let $\partial^1\Si$ be a union of some, at least one, connected components of $\partial\Si$ that:
\begin{itemize}
\item  are fixed pointwise by $\phi$, 
\item are contained in a single connected component of $\Si \setminus \calC$ and 
\item have positive fractional Dehn twist coefficient.
\end{itemize}
Then, there exists a mixed \tat graph $(\G^\bullet,A^\bullet,\delta_\bullet)$  with $A^0=\partial\Si\setminus\partial^1\Si$ and an embedding $\G \hookrightarrow \Si$ such that:

$$[\phi_{(\Sigma,\G)}]_{\partial\Si,\phi|_{\partial\Si}}= [\phi]_{\partial\Si,\phi|_{\partial\Si}}.$$

\end{theorem}

And we have as a corollary the case of the monodromy of plane branches: 
\begin{corollary}\label{cor:monod}
Let $f: \C^2 \to \C$ be an irreducible polynomial with an isolated singularity at $0$. Let $\Si$ be the corresponding Milnor fiber and let $h:\Si  \to \Si$ be a representative of the monodromy that fixes $\partial \Si$ pointwise. Then, there exists a mixed \tat graph $(\G^{\bullet}, \delta_{\bullet})$ (with no relative boundaries) embedded in $\Sigma$ such that $[\phi_{(\Sigma,\G)}]_{\partial \Si} = [h]_{\partial \Si}$.
\end{corollary}

\begin{proof}
In \cite{Camp2} it is given a description of the Milnor fiber and the monodromy, which, in particular shows that:
\begin{itemize}
\item[(0)] $h$ is pseudo-periodic,
\item[(1)] $G(h, \Si)$ is a tree, 
\item[(2)] it has  all screw numbers negative,
\item[(3)] the Milnor fiber has $1$ boundary component and the fractional Dehn twist coefficient with respect to it is positive.
\end{itemize}
hence we conclude by \Cref{thm:mixed_model}.
\end{proof}

\begin{figure}%
\includegraphics[width=90mm]{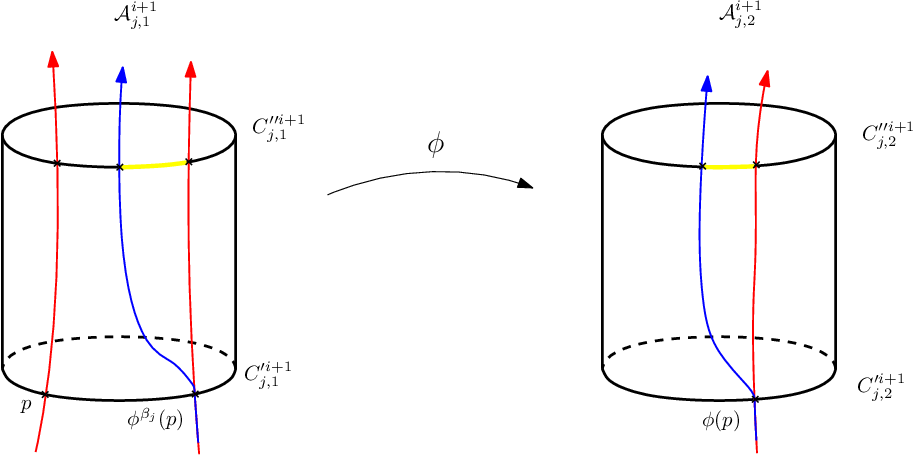}
\caption{In red, the edges of $\TG^{\leq i+1}$ passing by a point $p\in \Si^{\leq i}\cap \calA_1$, by $\phi(p)$ and by $\phi^{\beta_{i+1}}(p)$. In black, the image by $\phi_{\G^{i+1}}$ or $\hat{\phi}^{i+1}$ of the first one, and the image through the point $\phi^\beta(p)$ of an edge through the point $\phi^{\beta-1}(p)$. In yellow, a curve in $\calC_1''$ and $\calC_2''$ that has length $-s\cdot length (\calC_1'')$ and $-s/\alpha\cdot length (\calC_1'')$ respectively.} %
\label{fig:m_tat}
\end{figure}
\begin{remark} We could broaden the definition of mixed \tat twist allowing $\delta_i:\G^i\to \R$. In this way, we allow turning in the other direction along the separating annuli. We would find a \emph{signed} mixed \tat twist that would model pseudo-periodic homeomorphisms with positive screw numbers as well.
\end{remark}



\begin{example}\label{ex:example_thm}
Let $\Si$ be the surface of which is  symmetrically embedded in $\R^3$ as in \Cref{fig:ex_grafo} with its boundary component being the unit circle in the $xy$-plane. We follow the notation of the figure. 

Consider the rotation of $\pi$ radians around the $z$-axis. By the symmetric embedding of the surface, it leaves the surface invariant. Isotope the rotation so that it is the identity on $z\leq 0$ and its restriction to $\Si$ has fractional Dehn twist coefficient $1/2$ at the only boundary component. We denote the homeomorphism after the isotopy by $R_{\pi}$.

Let $D_i$ be a full positive Dehn twist on $\calA^{1}_{1,k}$, $k=1,2$. We define the homeomorphism \begin{equation}\label{eq:comp2}\phi:= D_2 \circ D_1^{-2} \circ R_{\pi}|_{\Si}\end{equation}

It is clear that $\phi$ is a pseudo-periodic  homeomorphism and it is already in quasi-canonical form of \Cref{re:stan}. We consider the decomposition of $\Sigma$ adapted to $\phi$ as in Figure \ref{fig:ex_grafo}. In particular we have a single annuli orbit $\calA^1$ and $s(\calA^1)=-1$.

One can  follow the proof of  \Cref{thm:mixed_model} to find a mixed \tat graph embedded in $\Si$ that models $\phi$. In  \Cref{fig:ex_grafo} you can find the abstract mixed \tat graph for $\delta_0=\pi$, $\delta_1=\pi/18$ which  determines the  homeomorphism only up to conjugacy (see \Cref{re:conj_mixed}).    
In \Cref{fig:ex_sur} you can find the embedded graph one obtains following the proof of \Cref{thm:mixed_model} which  determines also the isotopy classes $[\phi]_{\partial,\phi|_\partial}$ or $[\phi]$.

\begin{figure}[ht]
\centering
\includegraphics[scale=0.7]{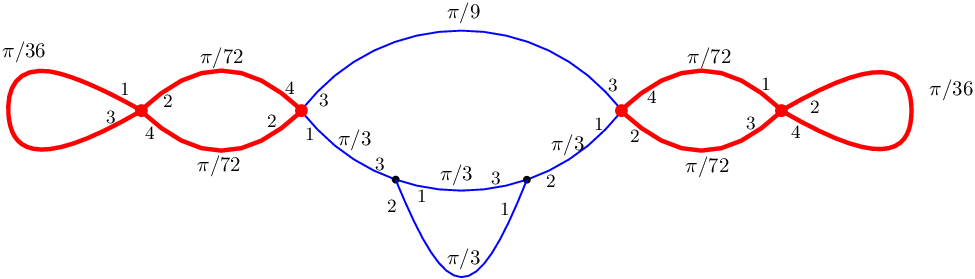}
\caption{Mixed \tat graph with $\G_1$ in red.}
\label{fig:ex_grafo}
\end{figure}

\begin{figure}[ht]
\centering
\includegraphics[scale=0.5]{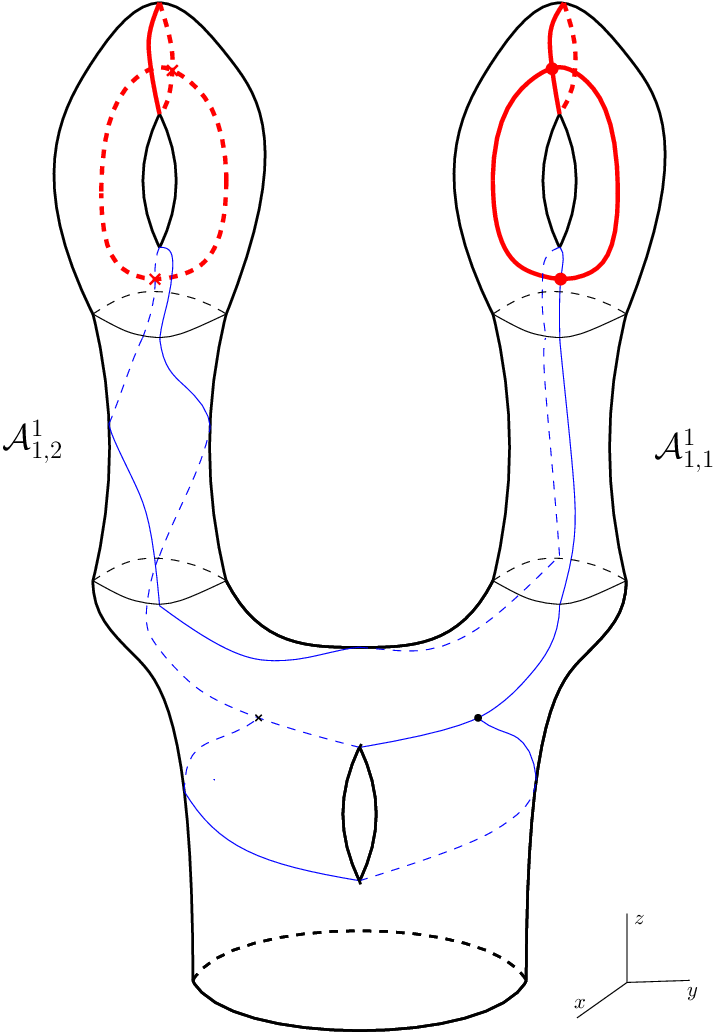}
\caption{We can see the surface $\Sigma$ of genus $3$ and $1$ boundary component. In red and blue we see the embedded graph $\G$ and in red we see $\G^1$.}

\label{fig:ex_sur}
\end{figure}

\end{example}

Note that in the key example in \Cref{sec:restricted}, the adjacency graph of 
the connected pieces of the subsurfaces where the restriction of the 
homeomorphims is periodic, is a tree. This was assumtion (1) in that section. 
In the following example we show a pseudo-periodic homeomorphism not satifying 
this restriction, that is a mixed \tat twist. As we have remarked at the 
introduction, in the sequel article \cite{Bal}, it is proved that this 
restriction is superfluous, and that mixed \tat twist is a much more general 
class.  

\begin{example}\label{ex:non_regular2} This is an example of mixed \tat twist that does not satisfy the assumption (1) in page \pageref{page:assump} at the beginning of the section.

It is induced by a mixed \tat graph with a filtration $\G \supset \G^1$. The thickening of $\G$ is a surface of genus $3$ and $1$ boundary component. 

We obtain the graph $\G$ as follows. We take the relative \tat graph of the first figure in \Cref{fig:ex2} that induces the rotation of the disk with 3 more boundary components. This will correspond to $\G_{\G^1}$. We also consider a graph $K_{3,3}$ with edges of length $\pi/12$. Note that the thickening of $K_{3,3}$ is a torus with 3 boundary components and that this metric $K_{3,3}$ is a \tat graph for $\pi/6$. 

Now, we identify the three relative 
components of $\G_{\G^1}$ with $\widetilde{K}_{3,3}$ which is obtained by cutting the graph $K_{3,3}$ along itself (and consists in three $\MS^1$). After the identification we quotient by the quotient map $\widetilde{K}_{3,3}\to K_{3,3}$ to obtain a graph $\G$ as the third one in \Cref{fig:ex2}. 
It is a mixed \tat graph for $\delta_0=\pi$ and $\delta_1=\pi/6$ with the embedding of $K_{3,3}$ as the level 1 subgraph $\G^1$.

It can be easily checked that the induced homeomorphism, that we denote by $\phi$, has restriction to $\Si^1_{1,1}$  isotopic to the rotation of \Cref{fig:ex2}, restriction to $\Si^{2}_{1,1}$ permuting the boundary components in the same way and the screw number of the annuli equals  $-1$.

\begin{figure}[!h]
\centering
\includegraphics[scale=0.65]{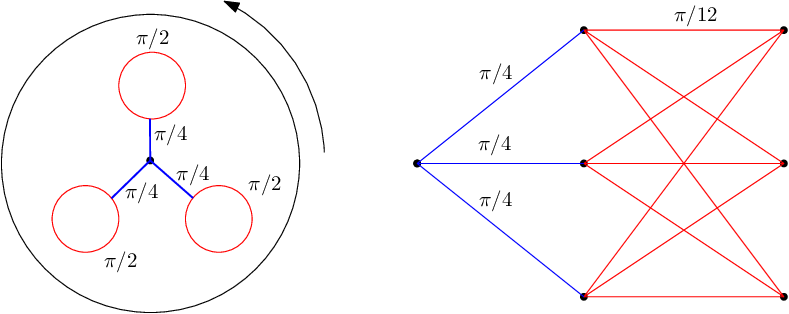}
\caption{On the left we see $\Si_{\G^1}$, in blue and red we see $\G_{\G^1}$ and in red we see $\TG^1$. On the right, we see a projection on the plane of the mixed \tat graph for $\delta_0=\pi$ and $\delta_1=\pi/6$ with $\G^1$ in red (note that $\G^1$ is $K_{3,3}$ with each edge of length $\pi/12$. If no cyclic order is indicated, then consider the counterclockwise ordering induced by the orientation of the plane.}
\label{fig:ex2}
\end{figure}

\begin{figure}[H]
\centering
\includegraphics[scale=0.55]{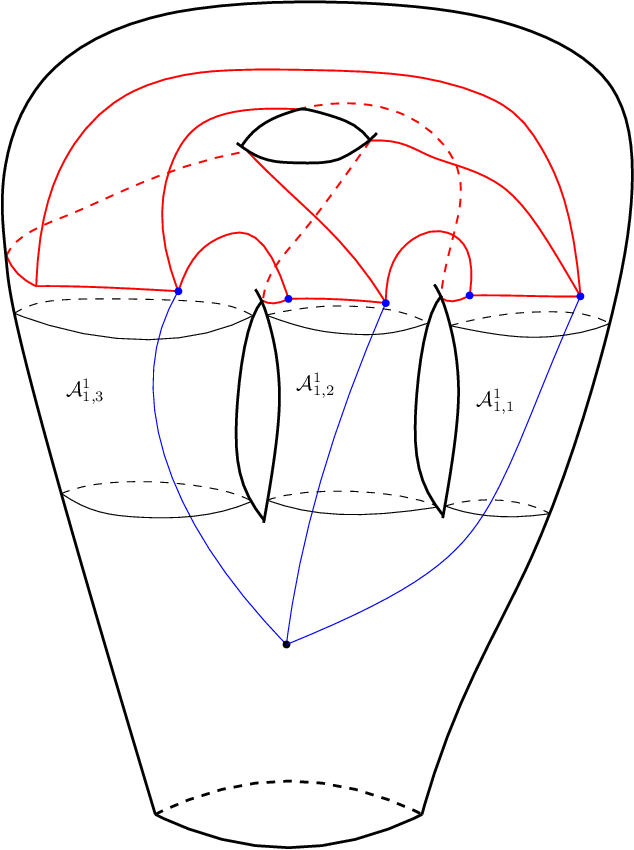}
\caption{The decomposition of $\Si$ adapted to $\phi$. The embedding of the mixed \tat graph for $\delta_0=\pi$ and $\delta_1=\pi/6$ with $\G^1$ in red. } 
\label{fig:ex2_sur}
\end{figure}
\end{example}

\bibliographystyle{alpha}
\bibliography{bibliography}

\begin{thebibliography}{{A'C}10}

\bibitem[A'C73]{Camp2}
Norbert A'Campo.
\newblock Sur la monodromie des singularit\'es isol\'ees d'hypersurfaces
  complexes.
\newblock {\em Invent. Math.}, 20:147--169, 1973.

\bibitem[{A'C}10]{Camp1}
Norbert {A'Campo}.
\newblock T{\^e}te-{\`a}-t{\^e}te twists and geometric monodromy.
\newblock Available at
  \url{http://www.math.sci.hiroshima-u.ac.jp/branched/files/2010/ACampo/ttt.pdf},
  2010.

\bibitem[FM12]{Farb}
Benson Farb and Dan Margalit.
\newblock {\em A primer on mapping class groups}, volume~49 of {\em Princeton
  Mathematical Series}.
\newblock Princeton University Press, Princeton, NJ, 2012.

\bibitem[Gra14]{Graf1}
Christian Graf.
\newblock T{\^e}te-{\`a}-t{\^e}te graphs and twists.
\newblock Aug 2014.
\newblock Available at \url{http://arxiv.org/abs/1408.1865v1}.

\bibitem[Gra15]{Graf}
Christian Graf.
\newblock T{\^e}te-{\`a}-t{\^e}te twists: Periodic mapping classes as graphs.
\newblock Jul 2015.

\bibitem[MMA11]{MM}
Yukio Matsumoto and Jos\'e~Mar\'ia Montesinos-Amilibia.
\newblock {\em Pseudo-periodic maps and degeneration of {R}iemann surfaces},
  volume 2030 of {\em Lecture Notes in Mathematics}.
\newblock Springer, Heidelberg, 2011.

\bibitem[Nie43]{Niel}
Jakob Nielsen.
\newblock Abbildungsklassen endlicher {O}rdnung.
\newblock {\em Acta Math.}, 75:23--115, 1943.

\bibitem[Nie44]{Niel1}
Jakob Nielsen.
\newblock Surface transformation classes of algebraically finite type.
\newblock {\em Danske Vid. Selsk. Math.-Phys. Medd.}, 21(2):89, 1944.

\bibitem[PS17]{Bal}
P.~{Portilla Cuadrado} and B.~{Sigurdsson}.
\newblock {Mixed t\^ete-\`a-t\^ete twists as monodromies associated with
  holomorphic function germs}.
\newblock {\em ArXiv e-prints}, December 2017.
\newblock Available at \url{https://arxiv.org/abs/1712.05988}.

\bibitem[RS77]{RayScott}
Frank Raymond and Leonard~L. Scott.
\newblock Failure of {N}ielsen's theorem in higher dimensions.
\newblock {\em Arch. Math. (Basel)}, 29(6):643--654, 1977.

\end{thebibliography}

\end{document}